\documentclass[12pt]{article}
\usepackage{graphicx}
\usepackage{amsmath,amsthm,amssymb,enumerate}%, esint}
\usepackage{euscript,mathrsfs}
\usepackage{color}
\usepackage[left=2cm,right=2cm,top=3.5cm,bottom=3.5cm]{geometry}
\usepackage{color}
\catcode`\@=11 \@addtoreset{equation}{section}

\catcode`\@=12

\allowdisplaybreaks

\newtheorem{Theorem}{Theorem}[section]
\newtheorem{Proposition}[Theorem]{Proposition}
\newtheorem{Lemma}[Theorem]{Lemma}
\newtheorem{Corollary}[Theorem]{Corollary}

\theoremstyle{definition}
\newtheorem{Definition}[Theorem]{Definition}

\newtheorem{Remark}[Theorem]{Remark}

\newcommand{\bTheorem}[1]{
\begin{Theorem} \label{T#1} }
\newcommand{\eT}{\end{Theorem}}

\newcommand{\bProposition}[1]{
\begin{Proposition} \label{P#1}}
\newcommand{\eP}{\end{Proposition}}

\newcommand{\bLemma}[1]{
\begin{Lemma} \label{L#1} }
\newcommand{\eL}{\end{Lemma}}

\newcommand{\bCorollary}[1]{
\begin{Corollary} \label{C#1} }
\newcommand{\eC}{\end{Corollary}}

\newcommand{\bRemark}[1]{
\begin{Remark} \label{R#1} }
\newcommand{\eR}{\end{Remark}}

\newcommand{\bDefinition}[1]{
\begin{Definition} \label{D#1} }
\newcommand{\eD}{\end{Definition}}

\newcommand{\Del}{\Delta_x}

\newcommand{\prst}{\mathbb{P}}

\newcommand{\ce}{c_{\ep}}

\newcommand{\tvre}{\widetilde{\vr}_{\ep}}
\newcommand{\tvue}{\widetilde{\vu}_{\ep}}
\newcommand{\tce}{\widetilde{c}_\ep}

\newcommand{\bfphi}{\boldsymbol{\varphi}}

\newcommand{\TN}{\mathcal{T}^N}

\newcommand{\blue}{\color{blue}}

\newcommand{\bFormula}[1]{
\begin{equation} \label{#1}}
\newcommand{\eF}{\end{equation}}

\newcommand{\Ov}[1]{\overline{#1}}

\newcommand{\DC}{C^\infty_c}

\newcommand{\aleq}{\stackrel{<}{\sim}}
\newcommand{\ageq}{\stackrel{>}{\sim}}

\newcommand{\vr}{\varrho}
\newcommand{\vre}{\vr_\ep}

\newcommand{\vue}{\vu_\ep}

\newcommand{\vu}{\vc{u}}

\newcommand{\vc}[1]{{\bf #1}}

\newcommand{\Div}{{\rm div}_x}
\newcommand{\Grad}{\nabla_x}

\newcommand{\dx}{\,{\rm d} {x}}

\newcommand{\dt}{\,{\rm d} t }

\newcommand{\intO}[1]{\int_{\Omega} #1 \ \dx}
\newcommand{\intTN}[1]{\int_{\TN} #1 \ \dx}

\newcommand{\D}{{\rm d}}
\newcommand{\ep}{\varepsilon}

\newcommand{\R}{\mathbb{R}}

\newcommand{\expe}[1]{ \mathbb{E} \left[ #1 \right] }

\newcommand{\vrn}{\vr_n}
\newcommand{\vun}{\vu_n}
\newcommand{\cn}{c_n}
\newcommand{\RR}{\mathbf{R}}
\definecolor{Cgrey}{rgb}{0.85,0.85,0.85}
\definecolor{Cblue}{rgb}{0.50,0.85,0.85}
\definecolor{Cred}{rgb}{1,0,0}
\definecolor{fancy}{rgb}{0.10,0.85,0.10}

\newcommand\Cbox[2]{%
    \newbox\contentbox%
    \newbox\bkgdbox%
    \setbox\contentbox\hbox to \hsize{%
        \vtop{
            \kern\columnsep
            \hbox to \hsize{%
                \kern\columnsep%
                \advance\hsize by -2\columnsep%
                \setlength{\textwidth}{\hsize}%
                \vbox{
                    \parskip=\baselineskip
                    \parindent=0bp
                    #2
                }%
                \kern\columnsep%
            }%
            \kern\columnsep%
        }%
    }%
    \setbox\bkgdbox\vbox{
        \color{#1}
        \hrule width  \wd\contentbox %
               height \ht\contentbox %
               depth  \dp\contentbox
        \color{black}
    }%
    \wd\bkgdbox=0bp%
    \vbox{\hbox to \hsize{\box\bkgdbox\box\contentbox}}%
    \vskip\baselineskip%
}

\date{}

\begin{document}

%%%%%%%%%%%%%%%%%%%%%%%%%%%%%%%%

\title{A diffuse interface model of a two--phase flow with thermal fluctuations}

\author{Eduard Feireisl
\thanks{The research of E.F.~leading to these results has received funding from the
Czech Sciences Foundation (GA\v CR), Grant Agreement
18--05974S. The Institute of Mathematics of the Academy of Sciences of
the Czech Republic is supported by RVO:67985840.} \and Madalina Petcu
}

\date{\today}

\maketitle

\bigskip

\centerline{Institute of Mathematics of the Academy of Sciences of the Czech Republic}

\centerline{\v Zitn\' a 25, CZ-115 67 Praha 1, Czech Republic}

\bigskip

\centerline{Laboratoire de Math\' ematiques et Applications, UMR CNRS 7348 - SP2MI}

\centerline{Universit\' e de Poitiers, Boulevard Marie et Pierre Curie - T\' el\' eport 2}

\centerline{86962 Chasseneuil, Futuroscope Cedex,
France}

\centerline{The Institute of Mathematics of the Romanian Academy, Bucharest, Romania}

\centerline{and}

\centerline{The Institute of Statistics and Applied Mathematics of the Romanian Academy, Bucharest, Romania}

\begin{abstract}

We consider a model of a two phase flow proposed by Anderson, McFadden and Wheeler taking into account possible thermal fluctuations.
The mathematical model consists of the compressible Navier--Stokes system coupled with the Cahn--Hilliard equation, where the latter is
driven by a multiplicative temporal white noise accounting for thermal fluctuations. We show that existence of
\emph{dissipative martingale solutions}
satisfying the associated total energy balance.

\end{abstract}


\tableofcontents

\section{Introduction}
\label{M}

The phase field approach has been widely used in models of multicomponent fluids to
avoid the mathematical difficulties related to the presence of sharp interfaces separating the fluid components. We consider a model of a binary
mixture of compressible viscous fluids proposed by Anderson, McFadden, and Wheeler \cite{AnFaWh}, where the time evolution of the phase variable -
the concentration difference - is described by means of the Cahn--Hilliard equation, cf. the related models and a general discussion
in Lowengrub and Truskinovski \cite{LoTr}. In addition to the standard formulation, we consider a stochastic perturbation of the
Cahn--Hilliard system that is nowadays a well accepted model of thermal fluctuations, see e.g. Debussche,
Gouden\` ege \cite{DebGou}, Gouden\` ege \cite{Gou2}, Gouden\` ege and Manca \cite{GouMan} or the more recent mathematical treatment
also including the fluid motion by Gal and Medjo \cite{GalMed1}, \cite{GalMed2}, \cite{Medj} or Scarpa \cite{Scar}.
Any mathematical model including the Navier--Stokes system necessarily inherits the well know difficulties concerning the existence of smooth solutions as well as related issues concerning well--posedness in general. Here we address the physically relevant 3D-setting in the context of general compressible viscous fluids, where the aforementioned difficulties are augmented by possible singular behavior of the fluid density.

\subsection{Field equations}

{ We start by considering a model of a binary mixture of compressible, viscous and macroscopically immiscible fluids, filling a bounded domain $\mathcal{D}\subset \mathbf{R}^N$.}
The basic field variables describing the mixture at a given time $t \in (0,T)$ and a spatial position $x \in \RR^N$, $N=1,2,3$ are:
the mass density $\vr = \vr(t,x)$, the macroscopic fluid velocity $\vu = \vu(t,x)$, and the order parameter $c = c(t,x)$
satisfying the following system of equations:
\begin{equation} \label{M1}
\D \vr + \Div (\vr \vu) \dt = 0,
\end{equation}
\begin{equation} \label{M2}
\begin{split}
\D (\vr \vu) + \Div (\vr \vu \otimes \vu) \dt &+ \Grad p(\vr,c) \dt = \Div \mathbb{S}(\Grad \vu) \dt -
\Div \left( \Grad c \otimes \Grad c - \frac{1}{2} |\Grad c |^2 \right) \dt,
\end{split}
\end{equation}
\begin{equation} \label{M3}
\D (\vr c ) + \Div (\vr c \vu) \dt = \Del \mu \dt + \vr \sigma(c) \D W.
\end{equation}

The pressure $p = p(\vr,c)$ and the chemical potential $\mu(\vr,c, \Del c)$ are derived from the
free energy
\begin{equation}\label{free}
E_{\rm free} = { \int_{\mathcal{D}}\vr f(\vr,c) + \frac{1}{2} |\Grad c|^2 \,{\rm d}x,}
\end{equation}
in the form of constitutive relations:
\[
\vr \mu = \vr \frac{\partial f(\vr, c)}{\partial c} - \Del c,\
p(\vr, c) = \vr^2 \frac{\partial f(\vr,c) }{\partial \vr}.
\]
The viscous stress tensor $\mathbb{S}$ is given by the standard Newton's rheological law,
\[
\mathbb{S}(\Grad \vu) = \nu_{\rm shear} \left( \Grad \vu^t + \Grad^t \vu - \frac{2}{N} \Div \vu \mathbb{I} \right)
+ \nu_{\rm bulk} \Div \vu \mathbb{I}.
\]
We refer to Anderson, McFadden, and Wheeler \cite{AnFaWh} for a detailed derivation as well as the physical background of the system (\ref{M1}--\ref{M3}) without stochastic forcing. 
The driving term in the Cahn--Hilliard equation accounts for the presence of thermal fluctuations and is interpreted as a stochastic integral in the sense of
It\^{o}, cf. \cite{DebGou}.

\subsection{Stochastic forcing}

Due to the presence of the stochastic integral, all quantities in (\ref{M1}--\ref{M3}) must be interpreted as \emph{random variables} with respect to a stochastic basis
$[\Omega, \mathfrak{B}, \mathcal{P}]$, where $\Omega$ is a probability space, $\mathfrak{B}$ a system of Borel sets, and $\mathcal{P}$ a probability
measure. $W$ is a cylindrical Wiener process $W = \{ \beta_k(t) \}_{k=1}^\infty$, where $\beta_k$ are standard real--valued mutually independent Wiener processes,
\begin{equation} \label{M8}
\begin{split}
\sigma (c) \D W &\equiv \sum_{k = 1}^\infty \alpha_k \sigma_k(c) \D \beta_k, \ \mbox{where}\ \sum_{k=1}^\infty \alpha^2_k < \infty,\\
\sigma_k &\in W^{2,\infty}(\RR),\ \| \sigma_k \|_{W^{2, \infty}(\RR)} \aleq 1 \ \mbox{uniformly in}\ k =1,2,\dots
\end{split}
\end{equation}
We denote by $\{ \mathfrak{F}_t \}_{t \geq 0}$ a complete right--continuous filtration in $\Omega$, non--anticipative with respect to $W$.
In accordance with our definition, $W$ may be seen as a continuous mapping of $[0,T]$ into the space $\RR^{\aleph_0}$ of sequences of real numbers.
We consider two subspaces of $\RR^{\alpha_0}$, namely $\mathfrak{U}$ and $\mathfrak{U}_0$
endowed with the Hilbert structure
\[
\left< \{ a_k \}; \{ b_k \} \right>_{\mathfrak{U}} = \sum_{k=1}^\infty a_k b_k,
\
\mbox{and} \
\left< \{ a_k \}; \{ b_k \} \right>_{\mathfrak{U}_0} = \sum_{k=1}^\infty \alpha_k^{ 2}  a_k b_k,
\]
respectively.

\subsection{Physical space, initial conditions}

To avoid technicalities, we restrict ourselves to the spatially periodic case, meaning the spatial domain ${\mathcal{D}}$ is represented by the flat torus
\[
\mathcal{T}^N = \left\{ [-\pi, \pi ]|_{\{ -\pi, \pi \}} \right\}^N, \ N=1,2,3.
\]
The system (\ref{M1}--\ref{M3}) is then considered on the space--time cylinder $Q_T \equiv (0,T) \times \TN$ and supplemented with the initial coditions
\begin{equation} \label{M9}
\vr(0, \cdot) = \vr_0, \ \vu(0,\cdot) = \vu_0, \ c(0, \cdot) = c_0.
\end{equation}
In agreement with the character of the problem, the initial data are random variables adapted (measurable) with respect to $\mathfrak{F}_0$. Among other technical conditions specified below, we assume that
\begin{equation} \label{M10}
\vr_0 > 0 \ \mbox{in}\ \TN,\ \intTN{ \vr_0 } = M > 0,\ \mathcal{P}\mbox{-a.s.}
\end{equation}
where the total mass $M$ is a \emph{deterministic} constant. As a consequence of our choice of the boundary conditions, we have
\[
\intTN{ \vr(\tau, \cdot) } = \intTN{ \vr_0 } = M \ \mbox{for any}\ \tau \geq 0.
\]

\subsection{Dissipative martingale solutions}

Our goal is to establish the existence of global-in-time weak martingale solutions to the problem (\ref{M1}--\ref{M3}), (\ref{M9}). Martingale solutions, in general, are defined on a different probability space than $\Omega$, whereas the process $W$ is considered as an integral part of the solution. Dissipative
martingale solution satisfy, in addition, a suitable form of the energy balance (inequality). The existence of martingale and dissipative martingale solutions
to a stochastic perturbation of the compressible Navier--Stokes system was proved by Breit and Hofmanov\' a \cite{BrHo} and later elaborated in the 
monograph \cite{BrFeHobook}. The existence of weak solutions to the deterministic version of the system (\ref{M1}--\ref{M3}), (\ref{M9}) was shown in \cite{AbFei}.

To the best of our knowledge, this is the first analytical treatment of the compressible Navier--Stokes system coupled with a stochastically driven
equation. Besides its physical relevance, the problem is therefore mathematically challenging. We propose a multi-level approximation scheme similar to
that one used in \cite{EF70}. More specifically, the equation of continuity (\ref{M1}) is regularized by means of an artificial viscosity term,
the momentum equation is approximated via a Faedo-Galerkin scheme, while the stochastic Cahn--Hilliard equation (\ref{M3}) is solved by an adaption of the
method proposed in the context of stochastic ODE's by It\^{o} and Nisio \cite{ItNo}. The solutions are obtained by means of the asymptotic limit
in the approximate scheme combining the deterministic theory for the Navier--Stokes system proposed by Lions \cite{LI4} and the stochastic framework in the context of weak topologies developed in \cite{BrFeHobook}.

The paper is organized as follows. Section \ref{P} contains some preliminary material; we introduce the concept of dissipative martingale solution and state
the main result. In Section \ref{A}, we introduce the basic approximation scheme to construct global-in-time solutions. Section \ref{S} is devoted to the analysis of the stochastic Cahn--Hilliard problem. The asymptotic limit in the approximation scheme is then performed in Section \ref{L}. Possible
extensions of the existence result are discussed in Section \ref{D}.

\section{Preliminaries, main results}
\label{P}

We start by specifying the basic hypotheses imposed on the constitutive functions.
We suppose the free energy $f = f(\vr,c)$ takes the form
\begin{equation} \label{H1}
\begin{split}
f(\vr, c) = f^e(\vr) &+ f^m(\vr, c) + f^c(c), \\ \mbox{where}\
f^e &= a \vr^{\gamma - 1}, \ a > 0, \ { \gamma >3,}\\
f^c &\in C^3(\RR), \ f^c(0) = 0,\ | \partial^j_c f^c | \aleq 1,\ j=2,3,\ \partial_c f^c (c) \approx c
\ \mbox{for}\ |c| > > 1 \ \mbox{large},\\
f^m(\vr, c) &= \log(\vr) H(c), \ H \in  C^3(\RR), \ | \partial^j_c H | \aleq 1, \ j=1,2,3.
\end{split}
\end{equation}
Here and hereafter, $A \aleq B$ means there is a positive constant $C$ such that $A \leq C B$, similarly $A \ageq B$ stands for
$A \geq C B$, and $A \approx B$ means $A \aleq B$ and $B \aleq A$. We refer to \cite{AbFei} for the physical background of such a choice of the free energy.

\subsection{Energy balance}

Assume, for { the moment}, that $[\vr, \vu, c]$ is a smooth solution of the system (\ref{M1}--\ref{M3}) in $Q_T$ { and that $\rho$ is a function bounded from below by some positive constant}. { In order to obtain the total energy balance for (\ref{M1}--\ref{M3}), we use the classical  It\^{o} formula (see e.g., \cite{LiRo}) applied to the total energy functional:
\begin{equation}\label{totalenerg}
\mathcal{E}=\int_{\mathcal{D}}\ \frac{1}{2} \vr |\vu|^2 + \vr f(\vr,c) + \frac{1}{2} |\Grad c|^2 \,{\rm d}x.
\end{equation}
Using the smoothness of the solution  $[\vr, \vu, c]$ and the Fr\'echet differentiability of the total energy $\mathcal{E}$, we obtain the following energy equation:
}
\begin{equation} \label{H2}
\begin{split}
\D &\intTN{ \left[ \frac{1}{2} \vr |\vu|^2 + \vr f(\vr,c) + \frac{1}{2} |\Grad c|^2 \right] } +
\intTN{ \Big[ \mathbb{S}(\Grad \vu) : \Grad \vu +  |\Grad \mu|^2 \Big] } \dt \\
&= \intTN{ \frac{1}{2} \sum_{k=1}^\infty \alpha_k^2 |\Grad \sigma_k (c) |^2 } \dt
+  \intTN{ \frac{1}{2} \vr \frac{\partial^2 f(\vr,c)}{\partial c^2} \sum_{k=1}^\infty \alpha^2_k |\sigma_k(c)|^2 } \dt\\
&+ \left( \intTN{ \vr \mu \sigma(c) } \right) \D W.
\end{split}
\end{equation}
The first two terms on the right-hand side of (\ref{H2}) result from the application of It\^{o}'s chain rule in the derivation. The theory of
\emph{dissipative} solutions includes (\ref{H2}), or the related inequality, as an integral part of the definition of a weak solution to
(\ref{M1}--\ref{M3}).

\subsection{Dissipative martingale solutions}

Martingale solutions are weak in both PDE and probability sense. This means that the partial derivatives are interpreted in the sense of distributions, while
the process $W$ as well as the associated probability space form a part of the unknowns of the problem. Accordingly, it is convenient to consider
the field variables as stochastic processes ranging in the abstract space of distributions. A suitable framework is provided by the scale of
Sobolev spaces $W^{\lambda, 2}(\TN)$, $\lambda \in \RR$.

\begin{Definition} \label{D1}

A \emph{martingale solution} of the system (\ref{M1}--\ref{M3}) is an object consisting of a complete probability basis
$[\Omega, \mathfrak{B}, \mathcal{P}] $, a cylindrical Wiener process $W$ with the associated right-continuous complete filtration
$\{ \mathfrak{F}_t \}_{t \geq 0}$, and a stochastic process $\tau \in [0,T] \mapsto [\vr(\tau, \cdot), \vu(\tau, \cdot), c(\tau, \cdot)]$
ranging in the Hilbert space $W^{-\ell,2}(\TN) \times W^{-\ell,2}(\TN; \RR^N) \times W^{-\ell, 2}(\TN)$, $\ell > N$ enjoying the following properties
$\mathcal{P}$-a.s.:
\begin{itemize}
\item $\vr \in C_{{\rm weak}}([0,T]; L^\gamma (\Omega))$, $\vr > 0$, $\vr \vu \in C_{{\rm weak}}([0,T]; L^{\frac{2 \gamma}{\gamma + 1}}(\Omega; \RR^N))$
are $\mathfrak{F}_t-$progressively measurable, and the integral identity
\begin{equation} \label{H3}
\left[ \intTN{ \vr \varphi } \right]_{t = 0}^{t = \tau} = \int_0^\tau \intTN{
\left[ \vr \partial_t \varphi + \vr \vu \cdot \Grad \varphi \right] } \dt
\end{equation}
holds
for any (deterministic) test function $\varphi \in C^1(\Ov{Q}_T)$;
\item $\vr c \in C_{\rm weak}([0,T]; L^2(\TN))$, $c \in L^\infty(0,T; W^{1,2}(\TN)) \cap L^2(0,T; W^{2,2}(\TN))$,\\
$\vu \in L^2(0,T; W^{1,2}(\TN; R^N))$ are $\mathfrak{F}_t-$progressively measurable, and the integral identity
\begin{equation} \label{H4}
\begin{split}
\left[ \intTN{ \vr \vu \cdot \bfphi } \right]_{t = 0}^{t = \tau} &=
\int_0^\tau \intTN{ \left[ \vr \vu \cdot \partial_t \bfphi + \vr \vu \otimes \vu : \Grad \bfphi + p(\vr,c) \Div \bfphi \right] } \dt\\
&- \int_0^\tau \intTN{ \left[ \mathbb{S}(\Grad \vu) : \Grad \bfphi - \left( \Grad c \otimes \Grad c - \frac{1}{2} |\Grad c|^2 \mathbb{I} \right)
: \Grad \bfphi \right] } \dt
\end{split}
\end{equation}
holds for any (deterministic) test function $\bfphi \in C^1(\Ov{Q}_T ; \RR^N)$;
\item $\mu \in L^2(0,T; W^{1,2}(\TN))$,
\begin{equation} \label{H5}
\vr \mu = \vr \frac{ \partial f(\vr, c) }{\partial c} - \Del c,
\end{equation}
and the integral identity
\begin{equation} \label{H6}
\begin{split}
&\left[ \intTN{ \vr c \varphi } \right]_{t=0}^{t = \tau} \\ &= \int_0^\tau \intTN{
\left[ \vr c \partial_t \varphi + \vr c \vu \cdot \Grad \varphi - \Grad \mu \cdot \Grad \varphi \right] } +
\int_0^\tau \left( \intTN{ \vr \sigma(c) \varphi } \right) \ \D W
\end{split}
\end{equation}
holds for any (deterministic) test function $\varphi \in C^1([0,T] \times \TN)$.

\end{itemize}
\end{Definition}

\begin{Remark} \label{R2}

Neither $\vu$ nor $c$ enjoy any continuity in time so it might seem ambiguous to speak about their progressive measurability. This must be understood in the sense that
\[
\vu, c \in L^2_{\rm prog} ( \Omega \times (0,T); L^2(\TN))
\]
meaning these functions are progressively measurable modulo a modification on the set of ${\rm prog}-$measure zero. Equivalently, any backward time regularization of these quantities is progressively measurable, see \cite[Chapter 2, Sections 2.1, 2.2]{BrFeHobook}.

\end{Remark}

As the probability basis is a part of the martingale solution, the initial data are attained only in law.

\begin{Definition} \label{D2}

Let $\Lambda_0$ be a Borel probability measure on the Hilbert space \\ $W^{-\ell,2}(\TN) \times W^{-\ell,2}(\TN; R^N) \times W^{-\ell,2}(\TN)$.
A martingale solution of the system (\ref{M1}--\ref{M3}) satisfies the initial conditions $[\vr_0, \vu_0, c_0]$ if
\[
\vr(0, \cdot) = \vr_0, \ \vr \vu(0, \cdot) = \vr_0 \vu_0, \ \vr c(0, \cdot) = \vr_0 c_0 \ \mathcal{P}-\mbox{a.s.},
\]
where
\[
\mathcal{L}[\vr_0, \vu_0, c_0] = \Lambda_0.
\]

\end{Definition}

Finally, we give a definition of the dissipative solution.

\begin{Definition} \label{D3}

A martingale solution of the system (\ref{M1}--\ref{M3}) is called \emph{dissipative martingale solution} if, $\mathcal{P}-a.s.$, the energy inequality
\begin{equation} \label{H8}
\begin{split}
& \left[ \psi \intTN{ \left[ \frac{1}{2} \vr |\vu|^2 + \vr f(\vr,c) + \frac{1}{2} |\Grad c|^2 \right] } \right]_{t = 0}^{t = \tau} +
\int_0^\tau \psi \intTN{ \Big[ \mathbb{S}(c, \Grad \vu) : \Grad \vu +  |\Grad \mu|^2 \Big] } \dt \\
&\leq \int_0^\tau \partial_t \psi \intTN{ \left[ \frac{1}{2} \vr |\vu|^2 + \vr f(\vr,c) + \frac{1}{2} |\Grad c|^2 \right] } \dt
\\
&+ \int_0^\tau \psi \intTN{ \frac{1}{2} \sum_{k=1}^\infty \alpha_k^2 |\Grad \sigma_k (c) |^2 } \dt
+ \int_0^\tau \psi \intTN{ \frac{1}{2} \vr \frac{\partial^2 f(\vr,c)}{\partial c^2} \sum_{k=1}^\infty \alpha^2_k |\sigma_k(c)|^2 } \dt\\
&+ \int_0^\tau \psi \left( \intTN{ \vr \mu \sigma(c) } \right) \D W
 \end{split}
\end{equation}
is satisfied for a.a. $\tau \in (0,T)$ and any (deterministic) $\psi \in C^1[0,T]$, $\psi \geq 0$. Here
\[
\left[ \frac{1}{2} \vr |\vu|^2 + \vr f(\vr,c) + \frac{1}{2} |\Grad c|^2 \right](0, \cdot) =
\left[ \frac{1}{2} \vr_0 |\vu_0|^2 + \vr_0 f(\vr_0,c_0) + \frac{1}{2} |\Grad c_0|^2 \right]
\]
where $[\vr_0, \vu_0, c_0]$ are the initial conditions { as defined in \eqref{D2}.}

\end{Definition}

At this stage, we are ready to formulate our main result.

\begin{Theorem} \label{Tmain}

Let $\sigma$ be given as in (\ref{M8}), and
let the function $f$ satisfy the hypothesis (\ref{H1}), with $\gamma > 3$. Let the initial law
$\Lambda_0$ enjoy the following properties:
\[
\Lambda_0 \left\{ \vr_0 > 0 \ \mbox{in}\ \TN, \ \intTN{\vr_0} = M \right\} = 1,
\]
where $M > 0$ is a deterministic constant,
\[
\mathbb{E}_{\Lambda_0} \left[ \left( \| \vr_0 \|_{L^\gamma(\TN)} + \| \vu_0 \|_{L^2(\TN; R^N)} + \| c_0 \|_{W^{1,2}(\TN)} \right)^\beta \right]
\leq c(\beta)
\]
for any $\beta > 0$.

Then the Navier--Stokes--Cahn--Hilliard system (\ref{M1}--\ref{M3}) admits a dissipative martingale solution with the initial law $\Lambda_0$.

\end{Theorem}

The rest of the paper is devoted to the proof of Theorem \ref{Tmain}. Possible extensions to a more general class of constitutive relations and/or
boundary conditions will be discussed in Section \ref{D}.

\section{Approximate solutions}
\label{A}

We start by introducing a family of approximate problems. Let $H_m$ be an $m-$dimensional space of trigonometric polynomials ranging
in $R^N$ endowed with the Hilbert structure of the space $L^2(\TN; \RR^N)$. For $\vu \in H_m$, $R > 0$, we introduce the cut--off operators
\[
[ \vu ]_R \equiv \chi \left( \| \vu \|_{H_m} - R \right) \vu,
\]
where
{
\[
\chi \in C^\infty(R),\  \chi(r) = \left\{ \begin{array}{l} 1 \ \mbox{for}\ r \leq 0, \\
 0 \ \mbox{for}\ r \geq 1,
\end{array} \right.
\]
and
\[
\chi'(r) \leq 0 \ \mbox{for}\ 0 \leq r \leq 1.
\]
}
Let $\Pi_m : L^2(\TN; \RR^N) \to H_m$ be the associated orthogonal projection.

The basic approximate problem reads as follows:
\begin{equation} \label{A5}
\D \vr + \Div (\vr [\vu]_R) \dt = \ep \Del \vr \dt,\ \vr(0, \cdot) = \vr_0;
\end{equation}
\begin{equation} \label{A6}
\begin{split}
\D &\Pi_m (\vr \vu) + \Div \Pi_m (\vr \vu \otimes [\vu]_R ) \dt + \chi \left( \| \vu \|_{H_m} - R \right) \Grad \Pi_m
\left[ p(\vr,c) + \sqrt{\ep} \vr^\alpha \right] \dt \\ &= \ep \Del \Pi_m (\vr \vu) \dt  + \Div \Pi_m \mathbb{S}(\Grad \vu) \dt -
\chi \left( \| \vu \|_{H_m} - R \right) \Div \Pi_m \left( \Grad c \otimes \Grad c - \frac{1}{2} |\Grad c |^2 \right) \dt,\\
&\Pi_m (\vr \vu)(0, \cdot) = \Pi_m (\vr_0 \vu_0);
\end{split}
\end{equation}
\begin{equation} \label{A7}
\D c + [\vc{u}]_R \cdot \Grad c \dt = \frac{1}{\vr} \Del \mu \dt + \sigma(c) \D W, \ \mu =  \frac{\partial f(\vr, c)}{\partial c} - { \frac{1}{\vr}}\Del c,\
c(0, \cdot) = c_0.
\end{equation}

In comparison with Theorem \ref{Tmain}, we impose rather restrictive conditions on the initial data:
\begin{equation} \label{A1}
\begin{split}
\Lambda_0 &\left\{ \vr_0 \in C^{2 + \nu},\  \intTN{ \vr_0} = M,\ 0 < C^1_\ep < \vr_0,\
\| \vr_0 \|_{C^{2 + \nu}(\TN)} \leq C^2_\ep \right\} = 1 \\ &\mbox{for some deterministic positive constants}\ C^1_\ep,\ C^2_\ep { \mbox{ and some } \nu \in (0,1)};\\
\mathbb{E}_{\Lambda_0} &\left[ \left( \| \vu_0 \|_{L^2(\TN; R^N)} + \| c_0 \|_{W^{2,2}(\TN)} \right)^\beta \right] \leq C_\ep (\beta)
\ \mbox{for all}\ \beta > 0.
\end{split}
\end{equation}

The equations (\ref{A5}), (\ref{A6}) are deterministic and can be solved pathwise. For a given
$c \in L^\infty(0,T; W^{1,2}(\TN)) \cap L^2(0,T; W^{2,2}(\TN))$, the problem (\ref{A5}), (\ref{A6}) admits a solution $[\vr, \vu]$ unique in the class
\begin{equation} \label{S1}
\begin{split}
\vu \in C^1([0,T]; H_m), \
\vr, \ \frac{1}{\vr} \in C([0,T]; C^{2 + \nu}(\TN)) \cap C^1([0,T]; C^{\nu}(\TN)),
\end{split}
\end{equation}
see \cite[Chapter 7]{EF70}. Moreover, thanks to the hypotheses imposed on the initial data, the functions $[\vu]_R$ and $\vr$ are bounded in the norms
corresponding to (\ref{S1}) by deterministic constants depending on $R$ and $m$.
In addition, it is easy to check that the mapping
\[
c \in L^\infty(0,T; W^{1,2}(\TN)) \cap L^2(0,T; W^{2,2}(\TN))
\mapsto [\vr, \vu] \in C([0,T]; C(\TN)) \times C([0,T]; H_m)
\]
is continuous. In particular, $[\vr, \vu]$ are $\mathfrak{F}_t-$progressively measurable as long as $c$ enjoys the same property.
In the next section, we discuss solvability of the stochastic Cahn--Hilliard system (\ref{A7}) for given $[\vu]_R$ and $\vr$.

\section{Solvability of the stochastic Cahn--Hilliard system}
\label{S}

We focus on the Cahn--Hilliard equation (\ref{A7}), where we assume, for a moment, that $\vr$ and $\vu$ are $\mathfrak{F}_t$-progressively measurable
processes fixed in the class (\ref{S1}).
In particular, the velocity field $[\vu]_R$ admits spatial derivatives of any order and $\vr$ is a smooth function
bounded below away from zero uniformly $\mathcal{P}-$a.s.
Our goal is to solve the stochastic equation
\begin{equation} \label{S2}
\D c + [\vc{u}]_R \cdot \Grad c \dt = \frac{1}{\vr} \Del \mu \dt + \sigma(c) \D W,
\ \mu = \frac{\partial f(\vr, c) }{\partial c} - \frac{1}{\vr} \Del c,
\ c(0, \cdot) = c_0.
\end{equation}

\subsection{A priori bounds}
\label{SS1}

We start with {\it a priori} bounds available for regular solutions of problem (\ref{S2}). Accordingly, the standard It\^{o}'s calculus can be applied
without modifications.

\subsubsection{Basic energy bound}

%The first step is to take the $L^2-$scalar product of (\ref{S2}) with $c$.
 Applying It\^{o}'s chain rule we get
\begin{equation} \label{S7}
\begin{split}
\D \intTN{ \frac{1}{2} c^2 } &= \intTN{ \frac{1}{2} \Div [\vu]_R c^2 } \dt + \intTN{ \frac{1}{\vr} \Del \mu c } \dt \\
&+ \intTN{ \frac{1}{2} \sum_{k=1}^\infty { \alpha_k^2} |\sigma_k(c) |^2 } \dt + \intTN{ c \sigma(c) } \ \D W.
\end{split}
\end{equation}
Furthermore, we rewrite
\[
\intTN{ \frac{1}{\vr} \Del \mu c } \dt = \intTN{ \left( \frac{ \partial f(\vr, c) }{\partial c} - \frac{1}{\vr} \Del c \right) \Del \frac{c}{\vr} } \dt,
\]
where
\[
\Del \frac{c}{\vr} = \frac{1}{\vr} \Del c - 2 \frac{ \Grad c \cdot \Grad \vr }{\vr^2} - c \Div \frac{\Grad \vr}{\vr^2}.
\]
Consequently, (\ref{S7}) gives rise to
\[
\begin{split}
\left[ \intTN{ c^2 } \right]_{t = 0}^{t = \tau} & + \int_0^\tau \intTN{ |\Del c |^2 } \dt \aleq \int_0^\tau \intTN{ c^2 } \dt\\
&+ \int_0^\tau \intTN{ |\Del c| \left( |c| + |\Grad c| \right) } \dt
+ \int_0^\tau \intTN{ \frac{\partial f(\vr, c) }{\partial c} \Del \frac{c}{\vr} }\dt\\ &+
\int_0^\tau \intTN{ \left( |c | + c^2 \right) } \ \D W.
\end{split}
\]

Next, using hypothesis (\ref{H1}) we obtain
\begin{equation} \label{S9}
\begin{split}
\left[ \intTN{ c^2 } \right]_{t = 0}^{t = \tau} & + \int_0^\tau \intTN{ |\Del c |^2 } \dt \aleq \int_0^\tau \intTN{ c^2 } \dt\\
&+ \int_0^\tau \intTN{ \left( |\Del c| + |\Grad c| \right) \left( |c| + |\Grad c| \right) } \dt
+
\int_0^\tau \intTN{ \left( |c | + c^2 \right) } \ \D W.
\end{split}
\end{equation}

Finally, we use the interpolation inequality
\[
\| \Grad c \|^2_{L^2(\TN)} \leq \delta \| \Del c \|_{L^2(\TN)}^2 + c(\delta) \| c \|_{L^2(\TN)}^2 \ \mbox{for any}\ \delta > 0,
\]
to deduce from (\ref{S9}) that
{ \begin{equation} \label{S11}
\begin{split}
\left[ \intTN{ c^2 } \right]_{t = 0}^{t = \tau} + \int_0^\tau \intTN{ |\Del c |^2 } \dt \aleq \int_0^\tau \intTN{ c^2 } \dt
+
\int_0^\tau \intTN{ \left( |c | + c^2 \right) } \ \D W.
\end{split}
\end{equation}
}
Thus passing to expectations in (\ref{S11}) and applying successively Burkholder--Davis--Gundy inequality and Gronwall's lemma we
may infer that
\begin{equation} \label{S12}
\expe{ \sup_{t \in [0,T]} \left( \intTN{ c^2 } \right)^\beta } + \expe{ \left( \int_0^T \intTN{ | \Del c |^2 } \dt \right)^{\beta} }
\aleq \expe{ \left( \intTN{ c_0^2 } \right)^\beta } \ \mbox{for any}\ \beta \geq 1.
\end{equation}

\subsubsection{{ First} order estimates}

{ In order to obtain higher order estimates for the concentration $c$, we apply the It\^o formula  to $|\nabla_x c|^2_{L^2(\TN)}.$}
After a routine manipulation, we obtain
\[
\begin{split}
\D \intTN{ \frac{1}{2} | \Grad c|^2 } &= \intTN{ \left( \frac{1}{2} \Div [\vu]_R |\Grad c|^2 - \Grad [\vu]_R : (\Grad c \otimes \Grad c ) \right)  } \dt
- \intTN{ \frac{1}{\vr} \Del \mu \Del c } \dt \\
&- \intTN{ \frac{1}{2} \sum_{k=1}^\infty { \alpha_k^2} \sigma_k(c) \Del \sigma_k(c)  } \dt - \intTN{ \sigma(c) \Del \sigma (c) } \ \D W,
\end{split}
\]
where
\[
\begin{split}
- \intTN{ \frac{1}{\vr} \Del \mu \Del c } &= \intTN{ \frac{1}{\vr} \Del \left( \frac{1}{\vr} \Del c \right) \Del c }
- \intTN{ \frac{1}{\vr} \Del \left( \frac{\partial f(\vr, c)}{\partial c} \right) \Del c }\\
&= - \intTN{ \left|  \Grad \left( \frac{1}{\vr} \Del c \right) \right|^2 }
+ \intTN{ \Grad \left( \frac{\partial f(\vr, c)}{\partial c} \right) \cdot \Grad \left( \frac{1}{\vr} \Del c \right) }.
\end{split}
\]

Thus, using hypotheses (\ref{M8}), (\ref{H1}) and repeating the arguments of the preceding section, we obtain
\begin{equation} \label{S13}
\expe{ \sup_{t \in [0,T]} \left( \intTN{ |\Grad c |^2  } \right)^\beta } + \expe{ \left( \int_0^T \intTN{ | \Grad \Del c |^2 } \dt \right)^{\beta} }
\aleq \expe{ \left( \intTN{ |\Grad c_0|^2 } \right)^\beta }
\end{equation}
for any $\beta \geq 1$.

\subsubsection{ Higher order estimates}

The final step consists in { applying the It\^o formula to  $|\Delta_x c|^2_{L^2(\TN)}$}:
\[
\begin{split}
\D \intTN{ \frac{1}{2} |\Del  c|^2 } &= - \intTN{ \Del (\Div [\vu]_R c ) \Del c } \dt + \intTN{ \frac{1}{\vr} \Del \mu \Del^2 c } \dt \\
&+ \intTN{ \frac{1}{2} \sum_{k=1}^\infty { \alpha_k^2} \sigma_k(c) \Del^2 \sigma_k (c) } \dt + \intTN{ \sigma(c) { \Del^2 c} }\ \D W.
\end{split}
\]

The most difficult term reads
\[
\intTN{ \frac{1}{\vr} \Del \mu \Del^2 c } = - \intTN{ \frac{1}{\vr} \Del \left( \frac{1}{\vr} \Del c \right) \Del^2 c } +
\intTN{ \frac{1}{\vr} \Del \left( \frac{\partial f(\vr,c) }{\partial c} \Del^2 c \right) },
\]
where
\[
\begin{split}
- &\intTN{ \frac{1}{\vr} \Del \left( \frac{1}{\vr} \Del c \right) \Del^2 c } \\ &\approx - \intTN{ \frac{1}{\vr^2} |\Del^2 c|^2 }
+ \intTN{ \frac{1}{\vr} D^2_x \left( \frac{1}{\vr} \right) \Del c \Del^2 c }
+ \intTN{ \frac{1}{\vr} D_x \vr (D_x \Del) c \Del^2 c }.
\end{split}
\]
Seeing that the third derivatives of $c$ are controlled by (\ref{S13}) we conclude
\begin{equation} \label{S14}
\expe{ \sup_{t \in [0,T]} \left( \intTN{ |\Del c |^2  } \right)^\beta } + \expe{ \left( \int_0^T \intTN{ | \Del^2 c |^2 } \dt \right)^{\beta} }
\aleq \expe{ \left( \intTN{ |\Del c_0|^2 } \right)^\beta }
\end{equation}
for any $\beta \geq 1$.

\subsection{Solvability of the approximate system}

Now we are ready to establish the existence of solutions to the approximate system (\ref{A5}--\ref{A7}). Unfortunately, this cannot be done in a direct fashion.
Instead, similarly to (\ref{A6}), the Cahn--Hilliard equation (\ref{S2}) will be approximated via a finite system of stochastic ODE's.

\subsubsection{Finite--dimensional approximation}

We replace (\ref{S2}) by
\begin{equation} \label{L2}
\D c + \Pi_n \left[ [\vu]_R \cdot \Grad c \right] \dt = \Pi_n \left[ \frac{1}{\vr} \Del \left( \frac{\partial f(\vr,c)}{\partial c} - \frac{1}{\vr} \Del c \right)
\right] \dt  + \Pi_n [\sigma (c) ] \D W, \ c(0, \cdot) = \Pi_n c_0,
\end{equation}
where $\Pi_n$ is the $L^2-$orthogonal projection onto a finite dimensional space $X_n$ formed by trigonometric polynomials of finite order.
In particular,
$\Pi_n$ commutes with spatial derivatives of any order.
Our goal is to solve the system (\ref{A5}), (\ref{A6}), (\ref{L2}) for fixed $m$, $n$, $R > 0$, and
$\ep > 0$. To this end we use the abstract result due to It\^{o} and Nisio \cite{ItNo}.

For a fixed $c \in C([0,T]; X_n)$, we consider the unique solution $[\vr, \vu] \equiv \vc{Q} [c]$ of the problem
\begin{equation} \label{L3}
\D \vr + \Div (\vr [\vu]_R) \dt = \ep \Del \vr \dt, \ \vr (0, \cdot) = \vr_0,
\end{equation}
\begin{equation} \label{L4}
\begin{split}
\D \Pi_m (\vr \vu) + \Div \Pi_m (\vr \vu \otimes [\vu]_R ) \dt &+ \chi \left( \| \vu \|_{H_m} - R \right) \Grad \Pi_m
\left[ p(\vr,c)  + \sqrt{\ep} \vr^\alpha \right] \dt\\ &= \ep \Del \Pi_m (\vr \vu) \dt  + \Div \Pi_m \mathbb{S}(c, \Grad \vu) \dt \\&-
\chi \left( \| \vu \|_{H_m} - R \right) \Div \Pi_m \left( \Grad c \otimes \Grad c - \frac{1}{2} |\Grad c |^2 \right) \dt,\\
\Pi_m (\vr \vu)(0, \cdot) &= \Pi_m (\vr_0 \vu_0).
\end{split}
\end{equation}
{ Here $\alpha$ is a positive constant that will be chosen greater than $4$ for reasons  that will become obvious in Section~$5$. }As already pointed out, the system (\ref{L3}), (\ref{L4}) is deterministic and can be solved pathwise. As shown in \cite[Chapter 7]{EF70}, there exists a solution $\vr$, $\vu$, uniquely determined in the class
\begin{equation} \label{L5}
\vu \in C([0,T]; H_{ m}); \ \vr, { \frac{1}{\vr} }\in C^1([0,T]; C^\nu(\TN)) \cap C([0,T]; C^{2 + \nu}(\TN))
\end{equation}
by the initial data $\vu_0 \in L^2(\TN; R^N)$, $\vr_0 \in C^{2 + \nu}(\TN)$, { $\vr_0 \geq \underline{\vr}> 0$ with $\underline{\vr}$ some positive deterministic constant}, and the forcing represented by $c$.
Moreover, it can be shown that the mapping
\[
\vc{Q}: c \in C([0,T]; X_n) \mapsto [\vr, \vu] \in { [ C^1([0,T]; C^\nu(\TN)) \cap C([0,T]; C^{2 + \nu}(\TN))] \times  C([0,T]; H_{ m})
}
\]
is continuous. In addition, such a mapping is obviously \emph{casual},
meaning the values of $[\vr, \vu]$ on a time interval $[0, \tau]$, $0 \leq \tau \leq T$ depend only on the values of $c$ in $[0, \tau]$.

Consequently, the problem (\ref{L2}) can be written in the form
\begin{equation} \label{L2a}
\begin{split}
\left[ c \right]_{t=0}^{t= \tau} &+ \int_0^\tau \Pi_n \left[ [ Q^2 [c] ]_R \cdot \Grad c \right] \dt \\
&= \int_0^\tau \Pi_n \left[ \frac{1}{Q^1[c]} \Del \left( \frac{\partial f(Q^1[c],c)}{\partial c} - \frac{1}{Q^1[c]} \Del c \right)
\right] \dt  + \int_0^\tau \Pi_n [\sigma (c) ] \D W.
\end{split}
\end{equation}
Thanks to the properties of the cut--off operators $[\cdot]_R$ and our choice of the initial data for $\vr$, we check easily that
\[
\left\| \Pi_n \left[ [ Q^2 [c] ]_R \cdot \Grad c \right] (\tau) \right\|_{X_n} +
\left\| \Pi_n \left[ \frac{1}{Q^1[c]} \Del \left( \frac{\partial f(Q^1[c],c)}{\partial c} - \frac{1}{Q^1[c]} \Del c \right)
\right](\tau) \right\|_{X_n} \leq { k(m, n, R)} \| c(\tau) \|_{X_n},
\]
{ where $k(m, n, R)$ is a positive constant that may depend on $m$, $n$ and $R$.}
Thus we may apply
the abstract result of It\^{o} and Nisio \cite[Theorem 1]{ItNo} on stochastic differential equations with memory to deduce the existence of a solution
$[\vr, \vu, c]$ of the problem (\ref{L2}), (\ref{L3}), and (\ref{L4}).

\begin{Proposition} \label{P1}
Let $m$, $n$, $\ep > 0$, and $R > 0$ be given. Suppose that the distribution of the initial data $[\vr_0, \vu_0, c_0]$ is determined by the law
$\Lambda_0$ specified in (\ref{A1}).

Then the problem (\ref{L2}--\ref{L4}) admits a solution $[\vr, \vu, c]$ in the class
\[
\vu \in C^1([0,T]; H_m), \ \vr, \ \frac{1}{\vr} \in C^1([0,T]; C^\nu(\TN)) \cap C([0,T]; C^{2 + \nu}(\TN), \ c \in C([0,T]; X_n).
\]
In addition, the process $\tau \mapsto [\vr(\tau, \cdot), \vu (\tau, \cdot), c(\tau,\cdot)]$ is $\mathfrak{F}_t-$progressively measurable.

\end{Proposition}

\subsubsection{The Galerkin limit in the Cahn--Hilliard system}
\label{SUB}

The ultimate goal of this part is to perform the limit $n \to \infty$ in the system (\ref{L2}--\ref{L4})
to recover a solution of the approximate problem (\ref{A5}--\ref{A7}).
We denote $\{ [\vrn, \vun, \cn] \}_{n=1}^\infty$ the family of solutions, the existence of which is guaranteed by Proposition \ref{P1}.

To begin, we establish uniform estimates independent of $n \to \infty$.
As the projections $\Pi_n$ commute with the Laplace operator, we may reproduce the bounds obtained in Section \ref{SS1} also at the level of Galerkin approximations. Summing up (\ref{S12}--\ref{S14}) we therefore deduce
\begin{equation} \label{L6}
\expe{ \left( \sup_{t \in [0,T]} \| \cn \|^2_{W^{2,2}(\TN)} \right)^\beta } +
\expe{ \left( \int_0^T \| \cn \|^2_{W^{4,2}(\TN)} \right)^\beta } \aleq
\expe{ \| c_0 \|^{2\beta}_{W^{2,2}(\TN)}  },
\end{equation}
where $\aleq$ hides a constant depending on the parameters $R$, $m$, as well as on the higher order norms of the initial data $\vr_0$ guaranteed by our
choice of the initial law $\Lambda_0$. By the same token,
the densities $\vr_n$ remain bounded in the space
\begin{equation} \label{L7}
C^1([0,T]; C^\nu(\TN)) \cap C([0,T]; C^{ 2+\nu}(\TN)) ,\ \vr_n \geq C^1_\ep > 0,
\end{equation}
uniformly by a deterministic constant.

Next, $\vun$ being determined by (\ref{L4}) will satisfy
\begin{equation} \label{L8}
\sup_{t \in [0,T]} \left( \| \vun \|_{ H_m} + \| \partial_t \vun \|_{ H_m} \right)
\aleq \| \vu_0 \|_{L^2(\TN; R^N)} + \int_0^T \| \cn \|^2_{W^{4,2}(\TN)} \dt \ \mbox{a.s.}
\end{equation}

Finally, we need time regularity of the processes $\cn$. To this end, we use equation (\ref{L2}) to compute
\[
\begin{split}
\left[ \intTN{ \cn \phi } \right]_{t = \tau_1}^{t = \tau_2} &=
\int_{\tau_1}^{\tau_2} \intTN{\left[ \frac{1}{\vrn} \Del \left( \frac{ \partial f(\vrn, \cn) }{\partial \vr} - \frac{1}{\vrn}
\Del \cn \right) - [\vun]_R \cdot \Grad \cn \right] \phi } \dt\\
&+ \int_{\tau_1}^{\tau_2} \left( \intTN{ \sigma (\cn) \phi } \right) \D W \ \mbox{for any}\ \phi \in X_n.
\end{split}
\]
By Burkholder--Davis--Gundy inequality,
\[
\expe{ \sup_{\tau_1 \leq \tau \leq \tau_2} \left| \int_{\tau_1}^{\tau} \left( \intTN{ \sigma (\cn) \phi } \right) \D W \right|^\beta
} \leq c(\beta) \expe{ \left( \int_{\tau_1}^{\tau_2} \sum_{k=1}^\infty \alpha^2_k \left( \intTN{ \sigma_k (\cn) \phi } \right)^2  \right)^{\beta/ 2} }
\]
for any $\beta > 0$, where, furthermore,
\[
\expe{ \left( \int_{\tau_1}^{\tau_2} \sum_{k=1}^\infty \alpha^2_k \left( \intTN{ \sigma_k (\cn) \phi } \right)^2  \right)^{\beta/ 2} }
\leq |\tau_2 - \tau_1 |^{\beta/2} \| \phi \|^{ \beta}_{L^2(\TN)} \expe{ { \left(\sum_{k=1}^\infty \alpha_k^2\right)^{\beta/2} }\sup_{ \tau \in [0,T] }
\| \sigma_k (\cn ) \|^\beta_{L^2(\TN)} }.
\]
Consequently, we can use the bounds imposed by $\Lambda_0$ and apply Kolmogorov continuity criterion to conclude that
$\cn \in C^{\omega}([0,T]; L^2(\TN))$
with
\begin{equation} \label{L9}
\expe{ \| \cn \|_{C^\omega([0,T]; L^2(\TN))}^\beta } \leq c(\beta), \ \beta > 2 \ \mbox{for a certain}\ 0 < \omega < \frac{1}{2}.
\end{equation}

With the bounds established in (\ref{L6}), (\ref{L7}), (\ref{L8}), and (\ref{L9}) at hand, we apply the stochastic compactness method to
perform the limit $n \to \infty$. The basic tool is the following extension of the classical Skorokhod theorem by Jakubowski \cite{Jakub}.

\begin{Theorem} \label{T1}
Let $\mathfrak{X}$ be a topological { sub-Polish} space, {meaning there exists} a countable family of continuous functions that separate points.
Let $\{ \mathfrak{L}_n \}_{n=1}^\infty$ be a tight sequence of Borel probability measures on $\mathfrak{X}$.

Then there { exists} a subsequence $\{ \mathfrak{L}_{n_k} \}_{k=1}^\infty$ and a sequence $\{ U_k \}_{k=1}^\infty$ of random variables
ranging in $\mathfrak{X}$ defined on the standard probability space $\left\{ [0,1], \mathfrak{B}[0,1], \dx \right\}$ such that
\begin{itemize}
\item the law of $U_k$ is $\mathfrak{L}_{n_k}$;
\item there {exists} a random variable $U$ such that $U_k(\omega) \to U(\omega)$ for any $\omega \in [0,1]$;
\item the law of $U$ is a Radon measure on $\mathcal{X}$.
\end{itemize}

\end{Theorem}

We consider a sequence of random variables $[\vrn, \vun, \cn, W]$ ranging in the pathspace $\mathfrak{X}$,
\[
\mathfrak{X} = \mathfrak{X}_\vr \times \mathfrak{X}_\vu \times \mathfrak{X}_c \times \mathfrak{X}_W,
\]
where
\[
\mathfrak{X}_\vr = C([0,T]; C^{2 + \mu}(\TN)) \cap C^\mu ([0,T] \times \TN),\
\mathfrak{X}_\vu = C^\mu ([0,T]; H_m),
\]
\[
\mathfrak{X}_c = C^\mu ([0,T]; L^2(\TN)) \cap \left[ L^\infty (0,T; W^{2,2}(\TN)); {\rm weak-}(*) \right] \cap
\left[L^2 (0,T; W^{4,2}(\TN)); {\rm weak} \right],
\]
\[
\mathfrak{X}_W = C([0,T]; \mathfrak{U}_0).
\]

Now it follows from the uniform bounds (\ref{L6}--\ref{L9}) that the sequence of joint laws\\
$\left\{ \mathfrak{L} [\vrn, \vun, \cn, W] \right\}_{n=1}^\infty$ is tight in $\mathfrak{X}$ for some $\mu > 0$. Applying Theorem \ref{T1}, we may pass to a subsequence,
change the probability space, and find a new family
of random variables $[\tilde{\vr}_n, \tilde{\vu}_n, \tilde{c}_n, \tilde{W}_n]$ converging to a limit $[\vr, \vu, c, W]$ in the topology of
$\mathfrak{X}$ a.s.

As the random variables $[\vrn, \vun, \cn, W]$ and $[\tilde{\vr}_n, \tilde{\vu}_n, \tilde{c}_n, \tilde{W}_n]$ have the same
distribution (law) in $\mathfrak{X}$, the the quantity $[\tilde{\vr}_n, \tilde{\vu}_n, \tilde{c}_n, \tilde{W}_n]$
satisfies the same system of equations (\ref{L2}--\ref{L4}), cf. \cite[Theorem 2.9.1]{BrFeHobook}.

To continue, we need the concept of \emph{random distribution} introduced in \cite[Chapter 2, Section 2.2]{BrFeHobook}. These are random variables
ranging in the space of distributions $\mathcal{D}'(Q_T)$, meaning the mapping
\[
\omega \in \Omega \to v(\omega) \in \mathcal{D}'(Q_T)
\]
is a random distribution, if the function $\left< v; \varphi \right>$ is $\mathcal{P}-$measurable for any $\varphi \in \DC(Q_T)$.
We introduce the history $\{ \sigma_t [v] \}_{t \geq 0}$ of a random distribution $v$ as the smallest complete right--continuous filtration such that
\[
\left\{ \omega \in \Omega \ \Big| \ \left< v, \varphi \right> \leq \lambda, \ \lambda \in R,\
\varphi \in \DC(Q_{\tau}) \right\} \subset \sigma_\tau \ \mbox{for any}\ \tau \in [0,T].
\]

To perform the limit in the stochastic integral, we report the following lemma, see \cite[Lemma 2.6.6]{BrFeHobook}:

\begin{Lemma} \label{LL1}
Let $( \Omega, \mathfrak{F}, \mathcal{P} )$
be a complete probability space. let $W_n$ be an $(\mathfrak{F}^n_t)$-cylindrical Wiener process and
let ${G}_n$ be an  $( \mathfrak{F}^n_t)$-progressively measurable stochastic process satisfying
\[
G_n\in L^2(0,T; L_2(\mathfrak{U};W^{\ell,2}(\TN))) \ \mbox{a.s.}
\]
Suppose that
\[
W_n \to W \ \mbox{in}\ C([0,T]; \mathfrak{U}_0) \ \mbox{in probability},
\]
\[
{G}_n \to {G} \ \mbox{in}\ L^2(0,T; L^2(\mathfrak{U}; W^{\ell,2}(\TN))) \  \mbox{in probability},
\]
where $W = \left\{ \beta_k \right\}_{k=1}^\infty$.
Let $( \mathfrak{F}_t )_{t \geq 0}$ be the filtration given as
\[
\mathfrak{F}_t = \sigma \big( \cup_{k = 1}^ \infty \sigma_t[ G_k ] \cup \sigma_t [\beta_k] \big).
\]

Then, after a possible change on a set of zero measure in $\Omega \times (0,T)$,  ${G}$ is
$( \mathfrak{F}_t )$-progressively measurable, and
\[
\int_0^\cdot {G}_n \, \D W_n \to \int_0^\cdot {G} \, \D W \ \mbox{in}\ L^2(0,T; W^{\ell,2}(\TN)) \ \mbox{in probability.}
\]

\end{Lemma}

With Lemma \ref{LL1} and the compactness properties of the space $\mathcal{X}$ at hand, it is not difficult to perform the limit
$n \to \infty$ in the sequence $[\tilde{\vr}_n, \tilde{\vu}_n, \tilde{c}_n, \tilde{W}_n]$. Besides the equations (\ref{L3}), (\ref{L4}), the
limit satisfies also (\ref{L2}) for \emph{any} $n=1,2,\dots$, meaning it satisfies the SPDE (\ref{A7}).
We have shown the following result:

\begin{Proposition} \label{P2}

Let $m$, $\ep > 0$, and $R > 0$ be given.
Suppose that the distribution of the initial data $[\vr_0, \vu_0, c_0]$ is determined by the law
$\Lambda_0$ specified in (\ref{A1}).

Then the approximate problem (\ref{A5}--\ref{A7}) admits a \emph{martingale} solution $[\vr, \vu, c, W]$ in the class
\[
\vu \in C^1([0,T]; H_m), \ \vr, \ \frac{1}{\vr} \in C^1([0,T]; C^\nu(\TN)) \cap C([0,T]; C^{2 + \nu}(\TN)),
\]
\[
\expe{ \left( \sup_{t \in [0,T]} \| c \|^2_{W^{2,2}(\TN)} \right)^\beta } +
\expe{ \left( \int_0^T \| c \|^2_{W^{4,2}(\TN)} \right)^\beta } \aleq 1, \ \beta \geq 0.
\]

\end{Proposition}

\begin{Remark} \label{RRR1}

\emph{Martingale solution} to the approximate system (\ref{A5}--\ref{A7}) is defined similarly to Definition \ref{D1}.

\end{Remark}

\section{Asymptotic limit }
\label{L}

Our ultimate goal is to perform successively the limits
$R \to \infty$, $m \to \infty$, and $\ep \to 0$ in the approximate system (\ref{A5}--\ref{A7}).

\subsection{Basic energy estimate}

To perform the asymptotic limit we need uniform estimates on the sequence of approximate solutions.
To this end, we
take the scalar product of the momentum equation
(\ref{A6}) with
$\vu \in C([0,T]; H_m)$ and use (\ref{A5}) obtaining
\begin{equation} \label{E1}
\begin{split}
\D \intTN{ \frac{1}{2} \vr |\vu|^2 } &+ \ep \intTN{ \vr |\Grad \vu|^2 } \dt
+ \intTN{ \mathbb{S}(\Grad \vu) : \Grad \vu } \dt  \\ &=  \intTN{ \vr^2 \frac{\partial f(\vr,c) }{\partial \vr} \Div [\vu]_R } \dt + \intTN{ \sqrt{\ep} \vr^\alpha \Div [\vu]_R } \dt \\
  &+
\intTN{ \left( \Grad c \otimes \Grad c - \frac{1}{2} |\Grad c|^2 \mathbb{I} \right) : \Grad [\vu]_R } \dt.
\end{split}
\end{equation}

The next step is {to apply the It\^o formula to the free energy defined in \eqref{free}}. At the first step of approximation, equation (\ref{A7}) admits strong solutions and It\^{o}'s
calculus applies directly:
\begin{equation} \label{E2}
\begin{split}
\D &\intTN{ \left( \vr f(\vr, c) + \frac{1}{2} |\Grad c|^2 \right) } + \intTN{ |\Grad\mu|^2 } \dt \\&=
\intTN{ \left( \vr \frac{\partial f(\vr,c)}{\partial \vr} + f(\vr, c) \right) \left( \ep \Del \vr - \Div (\vr [\vu]_R ) \right) } \dt  - \intTN{ \vr \mu [\vc{u}]_R \cdot \Grad c } \dt
\\ &+ \intTN{ \vr \mu \sigma(c) } \D W + \intTN{ \frac{1}{2} \sum_{k=1}^\infty \alpha_k^2 |\Grad \sigma_k (c) |^2 } \dt
+ \intTN{ \frac{1}{2} \vr \frac{\partial^2 f(\vr,c)}{\partial c^2} \sum_{k=1}^\infty { \alpha_k^2} |\sigma_k(c)|^2 } \dt.
\end{split}
\end{equation}

Furthermore,
\[
- \intTN{ \vr \mu [\vc{u}]_R \cdot \Grad c } \dt = \intTN{ \Del c [\vu]_R \cdot \Grad c }\, { {\rm d}t}-
\intTN{ \vr \frac{\partial f(\vr, c) }{\partial c} [\vu]_R \cdot \Grad c } \,{ {\rm d}t}.
\]
Seeing that
\[
\Del c \Grad c = \Div \left( \Grad c \otimes \Grad c - \frac{1}{2} |\Grad c|^2 \mathbb{I} \right)
\]
we may put together (\ref{E1}), (\ref{E2}) obtaining
\begin{equation} \label{E3}
\begin{split}
\D &\intTN{ \left[ \frac{1}{2} \vr |\vu|^2 + \vr f(\vr,c) + \frac{1}{2} |\Grad c|^2 \right] }  + \intTN{ \Big[ \mathbb{S}(\Grad \vu) : \Grad \vu +  |\Grad \mu|^2 \Big] } \dt \\ &+
\ep \intTN{ \vr |\Grad \vu|^2 } \dt\\
&=
\intTN{ \vr^2 \frac{\partial f(\vr,c) }{\partial \vr} \Div [\vu]_R } \dt + \intTN{ \sqrt{\ep} \vr^\alpha \Div [\vu]_R } \dt -
\intTN{ \vr \frac{\partial f(\vr, c) }{\partial c} [\vu]_R \cdot \Grad c } \, { {\rm d}t}\\
&-
\intTN{ \left( \vr \frac{\partial f(\vr,c)}{\partial \vr} + f(\vr, c) \right) \Div (\vr [\vu]_R ) } \dt
\\
&+ \ep \intTN{ \left( \vr \frac{\partial f(\vr,c)}{\partial \vr} + f(\vr, c) \right) \Del \vr   } \dt
\\
&+ \intTN{ \vr \mu \sigma(c) } \D W + \intTN{ \frac{1}{2} \sum_{k=1}^\infty \alpha_k^2 |\Grad \sigma_k (c) |^2 } \dt
+ \intTN{ \frac{1}{2} \vr \frac{\partial^2 f(\vr,c)}{\partial c^2} \sum_{k=1}^\infty \alpha_k^2 |\sigma_k(c)|^2 } \dt.
\end{split}
\end{equation}

Finally, we check by direct manipulation that
\[
\begin{split}
&\intTN{ \sqrt{\ep} \vr^\alpha \Div [\vu]_R } \dt +
\intTN{ \vr^2 \frac{\partial f(\vr,c) }{\partial \vr} \Div [\vu]_R } \dt -
\intTN{ \vr \frac{\partial f(\vr, c) }{\partial c} [\vu]_R \cdot \Grad c } \dt \\
&= - \D \intTN{ \frac{\sqrt{\ep}}{\alpha - 1} \vr^\alpha }  +
\intTN{ \left( \vr \frac{\partial f(\vr,c)}{\partial \vr} + f(\vr, c) \right) \Div (\vr [\vu]_R ) } \dt\\
&- \sqrt{\ep}\intTN{{\ep \alpha |\Grad \vr|^2 \vr^{\alpha-2}} }\dt;
\end{split}
\]
hence (\ref{E3}) reduces to
\begin{equation} \label{E4}
\begin{split}
\D &\intTN{ \left[ \frac{1}{2} \vr |\vu|^2 + \vr f(\vr,c) + \frac{1}{2} |\Grad c|^2 + \frac{\sqrt{\ep}}{\alpha - 1} \vr^\alpha \right] } + \intTN{ \Big[ \mathbb{S}(\Grad \vu) : \Grad \vu +  |\Grad \mu|^2 \Big] } \dt \\ &+
\ep \intTN{ \vr |\Grad \vu|^2 } \dt + \sqrt{\ep} \intTN{ { \ep \alpha \vr^{\alpha - 2} |\Grad \vr |^2} } \dt\\
&= -
\ep \intTN{ \frac{\partial^2 (\vr f(\vr, c)) }{\partial \vr^2} |\Grad \vr|^2   } \dt
- \ep \intTN{ \frac{\partial^2 (\vr f(\vr, c)) }{\partial \vr \partial c} \Grad \vr \cdot \Grad c } \dt
\\
&+ \intTN{ \vr \mu \sigma(c) } \D W + \intTN{ \frac{1}{2} \sum_{k=1}^\infty \alpha_k^2 |\Grad \sigma_k (c) |^2 } \dt
+ \intTN{ \frac{1}{2} \vr \frac{\partial^2 f(\vr,c)}{\partial c^2} \sum_{k=1}^\infty \alpha^2_k |\sigma_k(c)|^2 } \dt
 \end{split}
\end{equation}
{Remark moved at the beginning of 5.2.1}

Thus we have shown:

\begin{Proposition} \label{PEL1}

The martingale solutions obtained under the hypotheses of Proposition \ref{P2} satisfy the energy balance:
\begin{equation} \label{E4a}
\begin{split}
&\left[ \psi \intTN{ \left[ \frac{1}{2} \vr |\vu|^2 + \vr f(\vr,c) + \frac{1}{2} |\Grad c|^2 + \frac{\sqrt{\ep}}{\alpha - 1} \vr^\alpha \right] } \right]_{t = 0}^{\tau}
\\
&+
\int_0^\tau \psi \intTN{ \Big[ \mathbb{S}(\Grad \vu) : \Grad \vu +  |\Grad \mu|^2 \Big] } \dt \\ &+
\ep \int_0^\tau \psi \intTN{ \vr |\Grad \vu|^2 } \dt + \sqrt{\ep} \int_0^\tau \psi \intTN{ {\ep \alpha \vr^{\alpha - 2} |\Grad \vr |^2 }} \dt\\
&= \int_0^\tau \partial_t \psi \intTN{ \left[ \frac{1}{2} \vr |\vu|^2 + \vr f(\vr,c) + \frac{1}{2} |\Grad c|^2 + \frac{\sqrt{\ep}}{\alpha - 1} \vr^\alpha \right] } \dt\\
&-
\ep \int_0^\tau \psi \intTN{ \frac{\partial^2 (\vr f(\vr, c)) }{\partial \vr^2} |\Grad \vr|^2   } \dt
- \ep \int_0^\tau \psi \intTN{ \frac{\partial^2 (\vr f(\vr, c)) }{\partial \vr \partial c} \Grad \vr \cdot \Grad c } \dt
\\
&+ \int_0^\tau \psi \intTN{ \frac{1}{2} \sum_{k=1}^\infty \alpha_k^2 |\Grad \sigma_k (c) |^2 } \dt
+ \int_0^\tau \psi \intTN{ \frac{1}{2} \vr \frac{\partial^2 f(\vr,c)}{\partial c^2} \sum_{k=1}^\infty \alpha^2_k |\sigma_k(c)|^2 } \dt \\
&+ \int_0^\tau \psi \intTN{ \vr \mu \sigma(c) } \D W
 \end{split}
\end{equation}
for any $\psi \in C^1[0,T]$. Here
\[
\begin{split}
&\intTN{ \left[ \frac{1}{2} \vr |\vu|^2 + \vr f(\vr,c) + \frac{1}{2} |\Grad c|^2 + \frac{\sqrt{\ep}}{\alpha - 1} \vr^\alpha \right] { (0)}} \\
&=\intTN{ \left[ \frac{1}{2} \vr_0 |\vu_0|^2 + \vr f(\vr_0,c_0) + \frac{1}{2} |\Grad c_0|^2 + \frac{\sqrt{\ep}}{\alpha - 1} \vr_0^\alpha \right] }.
\end{split}
\]

\end{Proposition}

\subsection{Vanishing artificial viscosity limit}
\label{V}

We recall that our goal is to let $R \to \infty$, $m \to \infty$, and $\ep \to \infty$ in the approximate system (\ref{A5}--\ref{A7}). We focus only on the last and most difficult step letting $\ep \to 0$. We leave to the reader to work out the limits $R \to \infty$ and $m \to \infty$ that are easier and can be performed in a direct fashion. In addition, we also assume the initial data obey the law $\Lambda_0$ specified in Theorem \ref{Tmain}.

{ The weak formulation of the approximate problem}:

\begin{equation} \label{V1}
\D \vr + \Div (\vr \vu) \dt = \ep \Del \vr \dt, \ \vr(0, \cdot) = \vr_0;
\end{equation}
{reads as follows:}
\begin{equation} \label{V2}
\begin{split}
\left[ \intTN{ \vr \vu \cdot \bfphi } \right]_{t = 0}^{t = \tau} &=
\int_0^\tau \intTN{ \left[ \vr \vu \cdot \partial_t \bfphi + \vr \vu \otimes \vu : \Grad \bfphi + \left( p(\vr,c) + \sqrt{\ep} \vr^\alpha \right)\Div \bfphi \right] } \dt\\
&+ \int_0^\tau \intTN{ \ep \vr \vu \Del \bfphi - \mathbb{S}(\Grad \vu) : \Grad \bfphi } \dt\\ &+
\int_0^\tau \intTN{ \left( \Grad c \otimes \Grad c - \frac{1}{2} |\Grad c|^2 \mathbb{I} \right) : \Grad \bfphi } \dt,
\end{split}
\end{equation}
for any (deterministic) $\bfphi \in C^1([0,T] \times \TN; R^N)$;
\begin{equation} \label{V3}
\begin{split}
\left[ \intTN{ \vr c \varphi }  \right]_{t = 0}^{t = \tau} &=
\int_0^\tau \intTN{ \left[ \vr c \Big(\partial_t \varphi + \vu \cdot \Grad \varphi \Big) - \Grad \mu \cdot \Grad \varphi + \ep \Grad \vr  \cdot \Grad (c \varphi) \right] }\dt\\
&- \int_0^\tau \left( \intTN{ \vr \sigma(c) \varphi } \right) \D W
\end{split}
\end{equation}
for any (deterministic) $\varphi \in C^1([0,T] \times \TN)$;
\begin{equation} \label{V4}
\vr \mu = \vr \frac{\partial f(\vr,c)}{\partial c} - \Del c \ \mbox{a.a. in}\ (0,T) \times \Omega.
\end{equation}

\begin{Remark} \label{NR1}

At this stage, meaning having performed the limits $R \to \infty$ and $m \to \infty$, the density is still sufficiently regular to
satisfy equation (\ref{V1}) in the strong sense, cf. \cite[Chapter 7, Section 7.4]{EF70}.

\end{Remark}

In addition, we suppose solutions at this stage to satisfy the energy \emph{inequality}:
\begin{equation} \label{V5}
\begin{split}
& \left[ \psi \intTN{ \left[ \frac{1}{2} \vr |\vu|^2 + \vr f(\vr,c) + \frac{1}{2} |\Grad c|^2
+ \frac{\sqrt{\ep}}{\alpha - 1} \vr^\alpha
\right] } \right]_{t = 0}^{t = \tau} \\ &+
\int_0^\tau \psi \intTN{ \Big[ \mathbb{S}(\Grad \vu) : \Grad \vu +  |\Grad \mu|^2 \Big] } \dt \\
&\leq \int_0^\tau \partial_t \psi \intTN{ \left[ \frac{1}{2} \vr |\vu|^2 + \vr f(\vr,c) + \frac{1}{2} |\Grad c|^2 \right] } \dt
\\
& -
\ep \int_0^\tau \psi \intTN{ \frac{\partial^2 (\vr f(\vr, c)) }{\partial \vr^2} |\Grad \vr|^2   } \dt
- \ep \int_0^\tau \psi \intTN{ \frac{\partial^2 (\vr f(\vr, c)) }{\partial \vr \partial c} \Grad \vr \cdot \Grad c } \dt
\\
&+ \int_0^\tau \psi \intTN{ \frac{1}{2} \sum_{k=1}^\infty \alpha_k^2 |\Grad \sigma_k (c) |^2 } \dt
+ \int_0^\tau \psi \intTN{ \frac{1}{2} \vr \frac{\partial^2 f(\vr,c)}{\partial c^2} \sum_{k=1}^\infty \alpha^2_k |\sigma_k(c)|^2 } \dt\\
&+ \int_0^\tau \psi \left( \intTN{ \vr \mu \sigma(c) } \right) \D W
 \end{split}
\end{equation}
for a.a. $\tau \in (0,T)$ and any (deterministic) $\psi \in C^1[0,T]$, $\psi \geq 0$. Inequality (\ref{V6}) results from (\ref{E4a}) after
the limits $\R \to \infty$, $m \to \infty$. The details of this process will be clear from the arguments presented below.

\subsubsection{Uniform bounds}

The desired uniform bounds will be deduced from the energy inequality (\ref{V5}). In the following manipulation, we make use of the specific form of
the free energy,
\begin{equation} \label{V6}
f(\vr,c) = a \vr^{\gamma - 1} + \log(\vr) H(c) + f^c(c), \ a > 0,\ \gamma>3,
\end{equation}
satisfying hypothesis (\ref{H1}).
%{\blue $\gamma > 4$ or only $>3$ ??? it is $\alpha>4$ but we really need the same for $\gamma$?}
\begin{Remark} \label{R1}

In future considerations, we will have to control the term
\[
- \ep \intTN{ \frac{\partial^2 (\vr f(\vr, c)) }{\partial \vr^2} |\Grad \vr|^2  } \dt
\]
appearing on the right--hand side of the energy balance (\ref{E4}). Suppose that there is a function $\Gamma = \Gamma(\vr)$ such that
\[
\vr \mapsto \vr f(\vr, c) + \Gamma(\vr)
\]
is convex as a function of $\vr$ for any $c$. Accordingly,
\[
\begin{split}
- \ep \intTN{ \frac{\partial^2 (\vr f(\vr, c)) }{\partial \vr^2} |\Grad \vr|^2  } \dt &\leq
\ep \intTN{ \frac{\partial^2 \Gamma(\vr)}{\partial \vr^2} |\Grad \vr|^2 } \dt= - \ep \intTN{ \frac{\partial \Gamma(\vr)}{\partial \vr} \Del \vr } \dt \\ &=
- \D \intTN{ \Gamma(\vr) } - \intTN{ \frac{ \partial \Gamma(\vr)}{\partial \vr} \Div (\vr [\vu]_R) } \dt\\&=
- \D \intTN{ \Gamma(\vr) } + \intTN{ \left( \Gamma(\vr) - \vr \frac{\partial \Gamma(\vr)}{\partial \vr} \right) \Div [\vu]_R } \dt
\end{split}
\]

\end{Remark}

In view of (\ref{H1}), (\ref{V6}), we may { use  Remark \ref{R1} with $\Gamma(\vr)=\Gamma \vr \log(\vr) $ in order}  to estimate
\begin{equation} \label{V8}
- \ep \int_0^\tau \intTN{ \frac{\partial^2 (\vr f(\vr, c)) }{\partial \vr^2} |\Grad \vr|^2   } \dt
\leq - \Gamma \left[ \intTN{ \vr \log(\vr) } \right]_{t = 0}^{t = \tau}
- \Gamma \int_0^\tau \intTN{ \vr \Div \vu } \dt
\end{equation}
for some sufficiently large $\Gamma > 0$.

Next, we have to handle the integral
\[
\begin{split}
\ep &\left| \int_0^\tau \intTN{ \frac{\partial^2 (\vr f(\vr, c)) }{\partial \vr \partial c} \Grad \vr \cdot \Grad c } \dt \right| =
\ep \left| \int_0^\tau \intTN{ \left(1 + \log(\vr) \right) H'(c) \Grad \vr \cdot \Grad c } \dt \right| \\ &\aleq
\int_0^\tau \intO{ \frac{1}{2} |\Grad c |^2 } \dt + \ep^2 \int_0^\tau \intO{ (1 + |\log(\vr)|)^2 |\Grad \vr|^2 } \dt.
\end{split}
\]
To control the second integral, we use the renormalized version of (\ref{V1}), namely
\begin{equation} \label{V9}
\D b(\vr) + \Div (b(\vr) \vu) \dt  + \Big[ b'(\vr) \vr - b(\vr) \Big] \Div u\ \dt = \ep b'(\vr)  \Del \vr \ \dt.
\end{equation}
After integration, we obtain
\[
\ep \int_0^\tau \intTN{ b''(\vr) |\Grad \vr|^2 } \dt = - \left[ \intTN{ b(\vr) } \right]_{t = 0}^{t = \tau}
+ \int_0^\tau \intTN{ \Big[ b(\vr) - b'(\vr) \vr \Big] \Div \vu } \dt.
\]
Thus choosing $b_\omega(\vr) = \vr \log(\vr) + \vr^{2 + \omega}$ , $\omega > 0$, we obtain
\begin{equation} \label{V10}
\begin{split}
\ep &\left| \int_0^\tau \intTN{ \frac{\partial^2 (\vr f(\vr, c)) }{\partial \vr \partial c} \Grad \vr \cdot \Grad c } \dt \right|\\
&\aleq \int_0^\tau \intO{ \frac{1}{2} |\Grad c |^2 } \dt  - \ep \left[ \intTN{ b_{\omega}(\vr) } \right]_{t = 0}^{t = \tau}
+ \ep \int_0^\tau \intTN{ \Big[ b_\omega (\vr) - b'_\omega(\vr) \vr \Big] \Div \vu } \dt
\end{split}
\end{equation}

Finally, we recall Korn's inequality
\begin{equation} \label{V11}
\intTN{ \mathbb{S} (\Grad \vu) : \Grad \vu } \ageq \intTN{ |\Grad \vu|^2 }.
\end{equation}

Summing up (\ref{V8}), (\ref{V10}), (\ref{V11}) we deduce from (\ref{V5}):
\begin{equation} \label{V12}
\begin{split}
\Big[ &\intTN{ \left( \frac{1}{2} \vr |\vu|^2 + \vr f(\vr,c) + \frac{1}{2} |\Grad c|^2
+ \Gamma \vr \log(\vr) + \frac{\sqrt{\ep}}{\alpha - 1} \vr^\alpha  + \ep b_{\omega}(\vr)
\right) } \Big]_{t = 0}^{t = \tau} \\&+
\int_0^\tau \intTN{ \Big[ |\Grad \vu|^2 +  |\Grad \mu|^2 \Big] } \dt \\
&\aleq - \int_0^\tau \intTN{\Gamma \vr \Div \vu} \dt
+ \ep \int_0^\tau \intTN{ \Big[ b_\omega (\vr) - b'_\omega(\vr) \vr \Big] \Div \vu } \dt
\\
&+ \int_0^\tau \intTN{ \left[ |\Grad c|^2 + \vr |\log(\vr)| \right] } \dt
\\
&+ \int_0^\tau \left( \intTN{ \vr \mu \sigma(c) } \right) \D W
 \end{split}
\end{equation}
as long as $\sigma$ complies with (\ref{M8}).

As for the stochastic integral in (\ref{V12}), we apply Burkholder--Davis--Gundy inequality
\begin{equation} \label{V13}
\begin{split}
&\expe{ \sup_{t \in (0,\tau)} \left( \int_0^t \left( \intTN{ \vr \mu \sigma(c) } \right) \D W \right)^\beta }
\leq c(\beta) \expe{ \int_0^\tau \left( \sum_{k= 1}^\infty \alpha^2_k \left( \intTN{ \vr \mu \sigma_k (c) } \right)^2  \right)^{\beta/2} }\\
&\leq c(\beta) \expe{ \int_0^\tau \left( \sum_{k= 1}^\infty \alpha^2_k \left( \intTN{ \left( \vr f(\vr,c) \sigma_k(c) + \sigma'_k(c) |\Grad c|^2 \right) } \right)^2  \right)^{\beta/2} }\\
&\leq c(\beta) \expe{ \left( \sup_{t \in (0, \tau)} \intTN{ \vr |f(\vr,c)| + |\Grad c|^2 } \right)^{\beta} }.
\end{split}
\end{equation}

Finally, we observe that
\begin{equation} \label{V14}
\begin{split}
- \int_0^\tau \intTN{\Gamma \vr \Div \vu} \dt
&+ \ep \int_0^\tau \intTN{ \Big[ b_\omega (\vr) - b'_\omega(\vr) \vr \Big] \Div \vu } \dt \\ &\leq \frac{1}{2}
\int_0^\tau \intO{ |\Grad \vu|^2 } \dt + c \int_0^\tau \intO{ \left( \vr^2 + \ep \vr^{4 + 2 \omega} \right) } \dt.
\end{split}
\end{equation}
Thus the last term, being of order $\ep$, is dominated by the integral
$\sqrt{\ep}/(\alpha - 1) \intTN{ \vr^\alpha }$ as soon as $\alpha > 4 + 2 \omega$.

Summing up (\ref{V12}--\ref{V14}) and applying Gronwall's inequality, we deduce the following bounds:
\begin{equation} \label{V15}
\begin{split}
&\expe{ \sup_{\tau \in (0,T) } \left( \intTN{ \left[ \vr |\vu|^2 + \vr^\gamma + \vr c^2 + |\Grad c|^2 \right] } \right)^{\beta} }
+ \expe{ \left( \int_0^T \intTN{ |\Grad \vu|^2 + |\Grad \mu |^2 } \dt \right)^{\beta} } \\ &\leq
c(\beta) \left( 1 + \expe{ \left( \intTN{ \left[ \vr_0 |\vu_0|^2 + \vr_0^\gamma + \vr_0 c_0^2 + |\Grad c_0|^2 \right] } \right)^{\beta} } \right)
\ \mbox{for any}
\ \beta \geq 1.
\end{split}
\end{equation}

The inequality (\ref{V15}) yields bounds on $\Grad \vu$ and $\Grad \mu$. In order to control $\vu$ and $\mu$, a suitable version of Poincar\' e inequality is needed.

\begin{Lemma} \label{LL2}
Let $\Omega \subset \RR^N$, $N \geq 2$, be a bounded Lipschitz domain. Suppose that $\vr \in L^\gamma(\Omega)$, $\vr \geq 0$, $\intO{ \vr} \geq M > 0$,
where $\gamma > \frac{2N}{N+2}$.

Then
\[
\| v \|_{L^2(\Omega)}^2 \leq c(\Omega, \gamma, M) \left[ \left( 1 + \| \vr \|^2_{L^\gamma(\Omega)} \right)
\| \Grad v \|^2_{L^2(\Omega; \RR^N)} + \left| \intO{ \vr v } \right|^2 \right]
\]
for any $v \in W^{1,2}(\Omega)$.

\end{Lemma}

\begin{proof}

By standard Sobolev--Poincar\' e inequality, we have
\[
\| v - <v> \|_{L^q(\Omega)} \leq c(\Omega, q) \| \Grad v \|_{L^2(\Omega; \RR^N)}\ \mbox{whenever}\ 1 \leq q < \frac{2N}{N-2},
\]
where
\[
<v> = \frac{1}{|\Omega|} \intO{ v }.
\]
In particular,
\begin{equation} \label{V16}
\| v \|_{L^2(\Omega)} \leq c(\Omega) \left[ \| \Grad v \|_{L^2(\Omega; \RR^N)} + | < v > | \right].
\end{equation}

On the other hand,
\[
\begin{split}
|< v >| &\leq \frac{1}{M} \left| \intO{ \vr <v> } \right| = \frac{1}{M} \left| \intO{ \vr { (v - <v> ) }}  -\intO{ \vr v } \right| \\
& \leq \frac{1}{M} \left( \| \vr \|_{L^\gamma (\Omega)} \| v - <v> \|_{L^{\gamma'}(\Omega)} + \left| \intO{ \vr v } \right| \right)\\
&\leq \frac{c(\Omega; \gamma)}{M} \left( \| \vr \|_{L^\gamma (\Omega)} \| \Grad v \|_{L^2(\Omega)}  + \left| \intO{ \vr v } \right| \right),
\end{split}
\]
which, together with (\ref{V16}), completes the proof.

\end{proof}

Using Lemma \ref{LL2}, we can deduce from (\ref{V15}) that
\begin{equation} \label{V17}
\begin{split}
&\expe{ \left( \int_0^T \left( \| \vu \|^2_{W^{1,2}(\Omega; \RR^N)} + \| \mu \|^2_{W^{1,2}(\Omega)} \right)  \dt \right)^{\beta} } \\ &\aleq
c(\beta) \left( 1 + \expe{ \left( \intTN{ \left[ \vr_0 |\vu_0|^2 + \vr_0^\gamma + \vr_0 c_0^2 + |\Grad c_0|^2 \right] } \right)^{\blue \beta} } \right)
\end{split}
\end{equation}
provided the total mass is bounded below by a deterministic constant,
\begin{equation} \label{V18}
\intTN{ \vr } = \intTN{ \vr_0 } \geq M > 0 \ \mbox{a.s.}
\end{equation}
Indeed, we have
\[
\intTN{ \vr \mu } = \intTN{ \vr \frac{\partial f(\vr, c) }{\partial c} - \Del c } = \intTN{ \vr \frac{\partial f(\vr, c) }{\partial c} }
\]
bounded in terms of the initial data.

Finally, we follow the arguments used in Section \ref{SUB}, to deduce a bound on the modulus of continuity of $\vr c$. To this end, we use the weak formulation
(\ref{V3}), (\ref{V4}) of the Cahn--Hilliard equation. Given a test function $\phi \in C^1(\TN)$, the main task is to control the stochastic integral
\[
\int_{\tau_1}^{\tau_2} \left( \intTN{ \vr \sigma (c) \phi } \right) \ \D W.
\]
Similarly to Section \ref{SUB}, we use the Burkholder--Davis--Gundy inequality,
\[
\begin{split}
&\expe{ \sup_{\tau_1 \leq \tau \leq \tau_2} \left| \int_{\tau_1}^{\tau} \left( \intTN{ \vr \sigma (c) \phi } \right) \ \D W \right|^\beta } \\ &\leq c(\beta)
\expe{ \left( \int_{\tau_1}^{\tau_2} \sum_{k=1}^\infty {\alpha_k^2} \left( \intTN{ \vr \sigma_k (c) \phi } \right)^2 \right)^{\beta/2} },
\end{split}
\]
where, furthermore,
\[
\sup_{\tau_1 \leq  \tau \leq \tau_2} \left( \intTN{ \vr \sigma_k (c) \phi } \right)^2 \aleq \| \phi \|_{L^\infty(\TN)}^2 \left( \intTN{ \vr } \right)^2 =
\| \phi \|_{L^\infty(\TN)}^2 M^2.
\]
Using Kolmogorov's continuity criterion we may infer that
\begin{equation} \label{V19}
\expe{ \| \vr c \|_{C^\omega([0,T]; W^{-\ell,2} (\TN))}^\beta } \aleq c(\beta, \Lambda_0) \ \mbox{for a certain}\ 0 < \omega {\blue <} \frac{1}{2},\ \ell > \frac{N + 2}{2}.
\end{equation}

\begin{Remark} \label{NR3}

Recall that the test function $\phi$ was required to be continuously differentiable, meaning
\[
\phi \in W^{\ell,2}(\TN) \hookrightarrow C^1(\TN) \ \mbox{whenever} \ \ell > \frac{N + 2}{2}.
\]

\end{Remark}

\subsubsection{Stochastic compactness method}

Let $[\vre, \vue, \ce, \Grad \vue, \Grad \ce, \Del \ce, \mu_\ep ]$ be a family of approximate solutions of the problem (\ref{V1}--\ref{V4}) for $\ep > 0$.
{ In the previous section we obtained that the following sequences are uniformly bounded:
$$\sqrt{\vr_\ep}\vu_\ep \in L^\beta(\Omega, L^\infty(0,T, L^2(\TN))),$$
$$\vu_\ep \in L^\beta(\Omega, L^2(0,T, W^{1,2}(\TN))),$$
$$\mu_\ep \in L^\beta(\Omega, L^2(0,T, W^{1,2}(\TN))),$$
$$\vr_\ep \in L^\beta(\Omega, L^\infty(0,T, L^\gamma(\TN))),$$
$$\nabla_x c_\ep \in L^\beta(\Omega, L^\infty(0,T, L^2(\TN))),$$
$$\vr_\ep c_\ep \in L^\beta(\Omega, C^\omega([0,T]; W^{-\ell,2} (\TN))), \mbox{for  a certain } \omega\in(0,1/2) \mbox{ and } l>\dfrac{N+2}{2}.$$
Now, using these uniform bounds,} our ultimate goal is to apply a variant of the stochastic compactness method to obtain a sequence of approximate solutions converging in suitable topologies a.s.
To this end, it is convenient to work in weak topologies, in particular, we use the negative Sobolev spaces $W^{-\ell,2}(\TN)$.

Weak convergence in $L^p$-spaces can be effectively described in terms of Young measures that conveniently capture possible oscillations in non-linear compositions,
cf. e.g. Pedregal \cite{PED1}. The crucial tool to accommodate this technique
in the stochastic framework is the following result proved in \cite[Chapter 2.8]{BrFeHobook}:
\begin{Proposition} \label{P3}
Let
$\{ \Omega, \mathfrak{B}, \mathcal{P} \}$ be a complete probability space.
Let
\begin{itemize}
\item

$\{ \vc{U}_{0,\ep} \}_{\ep > 0}$ be a sequence of
random variables ranging in a Polish space $Y_0$;
\item
$\{ \vc{U}_\ep \}_{\ep > 0}$ a sequence of random variables in $L^1(Q_T; \RR^M)$;
\item
$\{ W_\ep \} _{\ep > 0}$
a sequence of cylindrical Wiener processes.
\end{itemize}

Let, finally,
$[[ \cdot ]]: W^{-m,2}(Q_T; \RR^M) \to [0, \infty]$, $m > \frac{N+1}{2}$ be a Borel measurable function.

Suppose that the family of laws of $\{ \vc{U}_{0,\ep } \}_{\ep > 0}$
is tight in $Y_0$. In addition, suppose that
for any $\delta > 0$, there exists $\Gamma > 0$ such that
\[
\begin{split}
\prst \left\{ \ \| \vc{U}_\ep   \|_{L^q(Q_T; \RR^M)} > \Gamma \ \right\} &< \delta \ \mbox{for some}\ q \geq 1;
\\
\prst \left\{ [[ \vc{U}_\ep ]] > \Gamma \ \right\} &< \delta
\end{split}
\]
uniformly for $\ep > 0$.

Then there exist subsequences of random variables $\{ \tilde{\vc{U}}_{0,\ep(j)} \}_{j \geq 1}$
in $Y_0$, $\{ \tilde{\vc{U}}_{\ep(j)} \}_{j \geq 1}$,\ $\tilde{\vc{U}}_{\ep (j)} \in L^1(Q_T; \RR^M)$ and
cylindrical Wiener processes $\{ \tilde{W}_{\ep (j)} \}_{j \geq 1}$ defined  on the standard probability space
$\Big\{ [0,1], \Ov{ \mathfrak{B}[0,1] }, {\rm d}y \Big\}$ enjoying the following properties ${\rm d}y$-a.s.:
\begin{itemize}
\item
\[
\left[ \vc{U}_{0,\ep (j)}, \vc{U}_{n(j)}, W_{\ep (j)} \right] \sim \left[ \tilde{\vc{U}}_{0, \ep (j)}, \tilde{\vc{U}}_{\ep (j)}, \tilde{W}_{\ep (j)} \right] \ \mbox{(equivalence in law)};
\]
\item
\[
\tilde{\vc{U}}_{\ep (j)} \to \vc{U} \ \mbox{in}\ W^{-m,2}(Q_T; \RR^M),\
g \left( \tilde{\vc{U}}_{\ep (j)} \right) \to \Ov{g (\vc{U})}\footnote{{Here and in all that follows, we denote the weak limit of nonlinear terms by the pointwise limit under the bar sign.}}
\ \mbox{weakly-(*) in}\ L^\infty(Q_T) \ \mbox{for any}\ g \in C_c(\RR^M);
\]
\item
\[
\tilde{\vc{U}}_{0, \ep(j)} \to \vc{U}_0 \ \mbox{in}\ Y_0,\ \tilde{W}_{\ep(j)} \to W \ \mbox{in}\ C([0,T]; \mathfrak{U}_0);
\]
\item
\[
\sup_{j \geq 1} [[ \tilde{\vc{U}}_{\ep (j)} ]] < \infty.
\]

\end{itemize}

If, in addition $q > 1$, then a.a.
\[
\tilde{\vc{U}}_{\ep(j)} \to \vc{U} \ \mbox{weakly in}\ L^q(Q_T; \RR^M),\
f \left( \tilde{\vc{U}}_{\ep (j)} \right) \to \Ov{f (\vc{U})}
\ \mbox{weakly in}\ L^r(Q_T)
\]
for any $f \in C(\RR^M)$ such that
\[
| f(\vc{v}) | \leq c \left(1 + |\vc{v}|^{s} \right),\ 1 \leq s < q,\ r = \frac{q}{s} > 1.
\]

\end{Proposition}

In the present case, the vector $\vc{U}_{0,\ep}$ will represent the initial data,
\[
\vc{U}_{0,\ep} = \left[ \vr_{0,\ep} , \vu_{0, \ep} , c_{0, \ep} \right],\ Y_0 = \left[ L^\gamma (\TN), L^2(\TN; \RR^N), W^{1,2}(\TN) \right],
\]
while $\vc{U}_\ep$ is the solution vector
\[
\vc{U}_\ep = \left[ \vre, \vue, \ce, \Grad \vue, \Grad \ce, \Del \ce, \mu_\ep \right].
\]
The functional $[[ \cdot ]]$ includes all norms estimated in (\ref{V15}), (\ref{V17}), (\ref{V19}), specifically,
\begin{equation} \label{V20}
\begin{split}
[[ \ \vr, \vu, c, \Grad \vu, \Grad c, \Del c, \mu  \ ]] &= {\rm ess} \sup_{\tau \in [0,T]} \left( \intTN{ \vr |\vu|^2 + \vr^\gamma + \vr c^2 + |\Grad c |^2 } \right)\\
&+ \int_0^T \left( \| \vu \|^2_{W^{1,2}(\TN; \RR^N)} + \| \mu \|^2_{W^{1,2}(\TN)} + \| \Del c \|^2_{L^q(\TN)} \right) \dt\\ &+
\sup_{\tau_1, \tau_2 \in [0,T]} \frac{ \| \vr c(\tau_1, \cdot) - \vr c (\tau_2, \cdot) \|_{W^{-\ell,2}(\TN)} } {| \tau_1 - \tau_2 |^\omega },\ 0 < \omega < \frac{1}{2},
\end{split}
\end{equation}
where
\[
1 < q < \gamma \ \mbox{if}\ N = 2, \ \mbox{and} \ q = \frac{6 \gamma}{\gamma + 6} \ \mbox{if}\ N=3.
\]
Indeed, relation (\ref{V4}) yields
\[
\Del c = \vr \left( \log(\vr) \partial_c H(c) + \partial_c f^c(c) \right) - \vr \mu;
\]
whence the bound for $\Del c$ follows from H\" older's inequality.

In view of Proposition \ref{P3}, we may change the probability space and obtain a family of random
variables
$[\tvre, \tvue, \tce, \Grad \tvue, \Grad \tce, \Del \tce, \widetilde{\mu}_\ep, \tilde{W}_\ep ]$ such that the following holds a.s.:
\begin{itemize}
\item
$[\vre, \vue, \ce, \Grad \vue, \Grad \ce, \Del \ce, \mu_\ep, W ]$ $\sim$ $[\tvre, \tvue, \tce, \Grad \tvue, \Grad \tce, \Del \tce, \widetilde{\mu}_\ep,
\tilde{W}_\ep ]$
(equivalence  in law);
\item
$[\tvre, \tvue, \tce, \Grad \tvue, \Grad \tce, \Del \tce, \widetilde{\mu}_\ep,
\tilde{W}_\ep ]$ satisfies the same system of equations (\ref{V1}--\ref{V4}), together with the energy inequality (\ref{V12});
\item
the sequence
$\left\{ [\tvre, \tvue, \tce, \Grad \tvue, \Grad \tce, \Del \tce, \widetilde{\mu}_\ep] \right\}_{\ep > 0} $ generates a Young measure;
\item
\[
\tvre \to \vr \ \mbox{weakly-(*) in}\ L^\infty(0,T; L^\gamma(\TN)) \ \mbox{and in}\
C([0,T]; W^{-\ell,2}(\TN));
\]
\[
\tvue \to \vu \ \mbox{weakly in}\ L^2(0,T; W^{1,2}(\TN; R^N));
\]
\[
\tvre \tvue \to \vr \vu \ \mbox{weakly-(*) in} \ L^\infty(0,T; L^{\frac{2 \gamma}{\gamma + 1}}(\TN; R^N))
\ \mbox{and in}\ C([0,T]; W^{-\ell,2}(\TN; R^N));
\]
\[
\tce \to c \ \mbox{weakly-(*) in}\ L^\infty(0,T; W^{1,2}(\TN))
\ \mbox{and weakly in}\ L^2(0,T; W^{2,2}(\TN));
\]
\[
\tvre \tce \to \vr c \ \mbox{in}\ C([0,T]; W^{-\ell,2}(\TN));
\]
\[
\tilde{\mu}_\ep \to \mu \ \mbox{weakly in}\ L^2(0,T; W^{1,2}(\TN));
\]
\[
\tilde{W}_\ep \to W \ \mbox{in}\ C([0,T]; \mathfrak{U}_0).
\]

\end{itemize}

Now, following the arguments of \cite{AbFei}, we show strong convergence of $\Grad \tce$. To begin with, { using the compactness results obtained previously,} we claim that
\[
\tvre \frac{\partial f(\tvre, \tce)}{\partial c} - \Del c_\ep =
\tvre \tilde{\mu}_\ep \to \vr \mu = \Ov{ \vr  \frac{ \partial f(\vr,c) }{\partial c} } - \Del c \ \mbox{weakly in} \ L^2(Q_T) \ \mbox{a.s.}
\]
Indeed, as $\gamma > 3$, it follows that
\[
\tvre \to \vr \in C([0,T]; W^{-1,2}(\TN)) \ \mbox{while}\ \tilde{\mu}_\ep \to \mu \ \mbox{weakly in}\ L^2(0,T; W^{1,2}(\TN)),
\]
which yields the claim.

Next, applying the same argument to $\tvre \tce$ and $\tilde{\mu}_\ep$, we get a.s.
\[
\tvre \tce \tilde{\mu}_\ep \to \vr c \mu \ \mbox{weakly in}\ L^s(Q_T) \ \mbox{for some}\ s > 1.
\]
In particular
\[
\int_0^T \intTN{ \tvre \tce \tilde{\mu}_\ep } \dt =
\int_0^T \intTN{ \tvre  \frac{\partial f(\tvre, \tce)}{\partial c} \tce + |\Grad \tce|^2 } \dt \to
\int_0^T \intTN{ \Ov{\vr \frac{\partial f(\vr, c) }{\partial c} } c + |\Grad c|^2 } \dt.
\]

Finally, we repeat the arguments of \cite[Section 2.6]{AbFei} to show
\[
\int_0^T \intTN{ \tvre  \frac{ \partial f(\tvre, \tce) }{\partial c} \tce } \dt \to \int_0^T \intTN{ \Ov{\vr  \frac{ \partial f(\vr, c) }{\partial c} } c } \dt,
\]
yielding the desired conclusion
\[
\Grad \tce \to \Grad c \ \mbox{in} \ L^2((0,T) \times \TN; \RR^N).
\]

At this stage, compactness of $\{ \tvre \}_{\ep > 0}$ can be established by means of equations (\ref{V1}), (\ref{V2}) using the same deterministic arguments as in \cite{AbFei}. Finally, we may perform the limit in the stochastic equation (\ref{V3}) using Lemma \ref{LL1}. The same can be done also for the energy inequality
(\ref{V5}) obtaining:
\[
\begin{split}
& \left[ \psi \intTN{ \left[ \frac{1}{2} \vr |\vu|^2 + \vr f(\vr,c) + \frac{1}{2} |\Grad c|^2 \right] } \right]_{t = 0}^{t = \tau} +
\int_0^\tau \psi \intTN{ \Big[ \mathbb{S}(\Grad \vu) : \Grad \vu +  |\Grad \mu|^2 \Big] } \dt \\
&\leq \int_0^\tau \partial_t \psi \intTN{ \left[ \frac{1}{2} \vr |\vu|^2 + \vr f(\vr,c) + \frac{1}{2} |\Grad c|^2 \right] } \dt
\\
&+ \int_0^\tau \psi \intTN{ \frac{1}{2} \sum_{k=1}^\infty \alpha_k^2 |\Grad \sigma_k (c) |^2 } \dt
+ \int_0^\tau \psi \intTN{ \frac{1}{2} \vr \frac{\partial^2 f(\vr,c)}{\partial c^2} \sum_{k=1}^\infty \alpha^2_k |\sigma_k(c)|^2 } \dt\\
&+ \int_0^\tau \psi \left( \intTN{ \vr \mu \sigma(c) } \right) \D W
 \end{split}
\]
for a.a. $\tau \in (0,T)$ and any (deterministic) $\psi \in C^1[0,T]$, $\psi \geq 0$.

We have proved Theorem \ref{Tmain}.

\section{Concluding remarks}
\label{D}

The hypothesis of Theorem \ref{Tmain} are slightly more restrictive than in the purely deterministic case examined in \cite{AbFei}, specifically,
$\gamma > 3$, $\nu_{\rm shear} = {\rm const}$, $\nu_{\rm bulk} = 0$, in contrast with
$\gamma > \frac{3}{2}$, $\nu_{\rm shear} = \nu_{\rm shear} (c) $, $\nu_{\rm bulk} = \nu_{\rm bulk} (c)$ for $N=3$ in \cite{AbFei}.
We claim that the conclusion of Theorem \ref{Tmain} holds under the more general constitutive hypothesis imposed in \cite{AbFei}. The approximation scheme must be augmented
by one level corresponding to the artificial pressure and the technique of \cite[Chapters 6,7]{EF70} must be used to overcome the problem of compactness in the Navier--Stokes system.
Moreover, the hypotheses imposed on the diffusion coefficients $\sigma_k$ can be relaxed as well. We leave to the interested reader to work out the details.

\def\cprime{$'$} \def\ocirc#1{\ifmmode\setbox0=\hbox{$#1$}\dimen0=\ht0
  \advance\dimen0 by1pt\rlap{\hbox to\wd0{\hss\raise\dimen0
  \hbox{\hskip.2em$\scriptscriptstyle\circ$}\hss}}#1\else {\accent"17 #1}\fi}

%\bibliographystyle{plain}
%\bibliography{citace}

\end{document}